\documentclass[a4paper,14pt]{amsart}

\usepackage[english,activeacute]{babel}

\usepackage[normalem]{ulem}
\usepackage{amsfonts}
\usepackage{amssymb}
\usepackage{amsthm}
\usepackage[english]{babel}
\usepackage{hhline}
\usepackage[ansinew]{inputenc}
\usepackage{amsmath}
\usepackage[all,2cell,ps]{xy}
\usepackage{qsymbols}
\usepackage{color}
\usepackage{epsfig}
\usepackage{graphics}
\usepackage{graphicx}%
\usepackage{enumerate}

\usepackage{amssymb}

\theoremstyle{plain}
\newtheorem{thm}{Theorem}
\newtheorem{cor}[thm]{Corollary}
\newtheorem{lem}[thm]{Lemma}
\newtheorem{prop}[thm]{Proposition}

\theoremstyle{definition}
\newtheorem{defn}[thm]{Definition}
\newtheorem{rem}[thm]{Remark}

\newtheorem{ex}[thm]{Example}
\newcommand{\Po}{\mathsf{P_0}}
\newcommand{\KPo}{\mathsf{KP}}
\newcommand{\KPop}{\mathsf{KP^{CK}}}
\newcommand{\DL}{\mathsf{BDL}}
\newcommand{\cKl}{\mathsf{KA_{\ce}}}
\newcommand{\cKlp}{\mathsf{KA_{\ce}^{CK}}}
\newcommand{\cN}{\mathsf{NA_{\ce}}}
\newcommand{\cNL}{\mathsf{NL_{\ce}}}
\newcommand{\He}{\mathsf{HA}}
\newcommand{\MS}{\mathsf{MS}}
\newcommand{\KMS}{\mathsf{KMS}}
\newcommand{\KMSp}{\mathsf{KMS^{CK}}}
\newcommand{\IS}{\mathsf{IS_{0}}}
\newcommand{\hIS}{\mathsf{hIS_{0}}}

\newcommand{\KhIS}{\mathsf{KhIS_{0}}}
\newcommand{\KhISp}{\mathsf{KhIS_{0}^{CK}}}
\newcommand{\KIS}{\mathsf{KIS_{0}}}
\newcommand{\Hil}{\mathsf{Hil_{0}}}
\newcommand{\KHil}{\mathsf{KHil_{0}}}
\newcommand{\KHilp}{\mathsf{KHil_{0}^{CK}}}
\newcommand{\hBDL}{\mathsf{hBDL}}
\newcommand{\KhBDL}{\mathsf{KhBDL}}
\newcommand{\KhBDLp}{\mathsf{KhBDL^{CK}}}
\newcommand{\SH}{\mathsf{SH}}
\newcommand{\KSH}{\mathsf{KSH}}
\newcommand{\K}{\mathrm{K}}
\newcommand{\C}{\mathrm{C}}

\newcommand{\ce}{\mathrm{c}}
\newcommand{\ra}{\rightarrow}
\newcommand{\Ra}{\Rightarrow}
\newcommand{\rla}{\leftrightarrow}
\newcommand{\we}{\wedge}
\newcommand{\CK}{\mathrm{CK}}

\newcommand{\Won}{\mathrm{(W1)}}
\newcommand{\Wtw}{\mathrm{(W2)}}
\newcommand{\Wth}{\mathrm{(W3)}}
\newcommand{\Ko}{\mathrm{(K1)}}
\newcommand{\Ktw}{\mathrm{(K2)}}
\newcommand{\Kth}{\mathrm{(K3)}}
\newcommand{\Kfo}{\mathrm{(K4)}}
\newcommand{\Kfi}{\mathrm{(K5)}}
\newcommand{\Ksi}{\mathrm{(K6)}}
\newcommand{\Kse}{\mathrm{(K7)}}
\newcommand{\KMo}{\mathrm{(KM1)}}
\newcommand{\KMtw}{\mathrm{(KM2)}}
\newcommand{\KMth}{\mathrm{(KM3)}}
\newcommand{\KMfo}{\mathrm{(KM4)}}
\newcommand{\KHon}{\mathrm{(KHil1)}}
\newcommand{\KHtw}{\mathrm{(KHil2)}}
\newcommand{\KHth}{\mathrm{(KHil3)}}
\newcommand{\KHfo}{\mathrm{(KHil4)}}
\newcommand{\KHfi}{\mathrm{(KHil5)}}
\newcommand{\SHon}{\mathrm{(SH1)}}
\newcommand{\SHtw}{\mathrm{(SH2)}}
\newcommand{\SHth}{\mathrm{(SH3)}}
\newcommand{\SHfo}{\mathrm{(SH4)}}
\newcommand{\KSHth}{\mathrm{(KSH3)}}
\newcommand{\Ton}{\mathrm{(C1)}}
\newcommand{\Ttw}{\mathrm{(C2)}}
\newcommand{\Tth}{\mathrm{(C3)}}
\newcommand{\Con}{\mathrm{Con}}
\newcommand{\Conwb}{\mathrm{Con_{wb}}}

\begin{document}

\title{On Kalman's functor for bounded hemi-implicative semilattices
and hemi-implicative lattices}

\author{Ramon Jansana and Hernan Javier San Mart\'{\i}n}

\maketitle

\begin{abstract}
Hemi-implicative semilattices (lattices), originally defined under
the name of weak implicative semilattices (lattices), were
introduced by the second author of the present paper. A
hemi-implicative semilattice is an algebra $(H,\we,\ra,1)$ of type
$(2,2,0)$ such that $(H,\we)$ is a meet semilattice, $1$ is the
greatest element with respect to the order, $a\ra a = 1$ for every
$a\in H$ and for every $a$, $b$, $c\in H$, if $a\leq b\ra c$ then
$a\we b \leq c$. A bounded hemi-implicative semilattice is an
algebra $(H,\we,\ra,0,1)$ of type $(2,2,0,0)$ such that
$(H,\we,\ra,1)$ is a hemi-implicative semilattice and $0$ is the
first element with respect to the order. A hemi-implicative
lattice is an algebra $(H,\we,\vee,\ra,0,1)$ of type $(2,2,2,0,0)$
such that $(H,\we,\vee,0,1)$ is a bounded distributive lattice and
the reduct algebra $(H,\we,\ra,1)$ is a hemi-implicative
semilattice.

In this paper we introduce an equivalence for the categories of
bounded hemi-implicative semilattices and hemi-implicative
lattices, respectively, which is motivated by an old construction
due J. Kalman that relates bounded distributive lattices and
Kleene algebras.
\end{abstract}

\smallskip
\noindent \textbf{Keywords:} Semilattices, distributive lattices,
implications, categorical equivalences, congruences.

\section{Introduction}\label{int}


Inspired by results due to J. Kalman relating to lattices
\cite{K}, R. Cignoli proved in \cite{cig} that a construction of
J. Kalman can be extended to a functor $\K$ from the category of
bounded distributive lattices to the category of Kleene algebras
and that this functor has a left adjoint \cite[Theorem 1.7]{cig}. He
also showed that there exists an equivalence between the category
of bounded distributive lattices and the full subcategory of
centered Kleene algebras whose objects satisfy a condition called
interpolation property \cite[Theorem 2.4]{cig}. Moreover, R.
Cignoli also proved that there exists an equivalence between the
category of Heyting algebras and the category of centered Nelson
algebras \cite[Theorem 3.14]{cig}. These results were extended by
J.L. Castiglioni, R. Lewin, M. Menni and M. Sagastume in the
context of residuated lattices \cite{camesa2, CLS}. On the other
hand, the original Kalman's construction  was also extended in
\cite{CCSM} by J.L. Castiglioni, S. Celani and the second author
of the present article to the framework of algebras with
implication $(H,\we,\vee, \ra, 0,1)$ which satisfy the following
additional condition: for every $a,b,c\in H$, if $a\leq b\ra c$
then $a\we b \leq c$. Algebras with implication were introduced by
S. Celani in \cite{Ce}.

A generalization of Heyting algebras is provided by the notion of
hemi-implicative semilattice (lattice), introduced in \cite{SM2}
under the name weak implicative semilattices (lattices). An
algebra $(H,\we,\ra,1)$ of type $(2,2,0)$ is said to be  a
\emph{hemi-implicative semilattice} if $(H,\we,1)$ is an upper
bounded semilattice \footnote{Let $(H,\leq)$ be a poset. If any
two elements $a$, $b \in H$ have a greatest lower bound (i.e., an
infimum), which is denoted by $a \we b$, then the algebra
$(H,\we)$ is called a \emph{meet semilattice}. Throughout this
paper we write \emph{semilattice} in place of meet semilattice. A
semilattice $(H,\we)$ is said to be \emph{upper bounded} if it has
a greatest element; in this case we write $(H,\we,1)$, where $1$
is the last element of $(H,\leq)$. A \emph{bounded semilattice} is
an algebra $(H,\we,0,1)$ of type $(2,0,0)$ such that $(H,\we,1)$
is an upper bounded semilattice and $0$ is the first element of
$(H,\leq)$. Frequently in the literature what we call upper
bounded semilattice is known as bounded semilattice.}, $a\ra a =
1$ for every $a\in H$ and for every $a,b,c\in H$, if $a\leq b\ra
c$ then $a\we b \leq c$. A \emph{bounded hemi-implicative
semilattice} is an algebra $(H,\we,\ra,0,1)$ of type $(2,2,0,0)$
such that $(H,\we,\ra,1)$ is a hemi-implicative semilattice and
$0$ is the first element with respect to the order. A
\emph{hemi-implicative lattice} is an algebra
$(H,\we,\vee,\ra,0,1)$ of type $(2,2,2,0,0)$ such that
$(H,\we,\vee,0,1)$ is a bounded distributive lattice and the
reduct algebra $(H,\we,\ra,1)$ is a hemi-implicative semilattice.
Implicative semilattices \cite{N} and Hilbert algebras with
infimum \cite{Fig} are examples of hemi-implicative semilattices.
Semi-Heyting algebras \cite{Sanka} and some algebras studied in
\cite{CCSM} are examples of hemi-implicative lattices. For
instance, the RWH-algebras, introduced and studied by S. Celani
and the first author of this article in \cite{CJ}, are examples of
hemi-implicative lattices.

The applications of Kalman's construction given in \cite{cig}
suggest that it is potentially fruitful to understand Kalman's
work in the context of bounded hemi-implicative semilattices and
hemi-implicative lattices. We do this in the present paper. The
main goal of the paper is to introduce and study an equivalence
for the categories of bounded hemi-implicative semilattices and
hemi-implicative lattices, respectively, and for some of its full
subcategories.

The paper is organized as follows. In Section \ref{Kc} we give
some results about Kalman's functor for bounded distributive
lattices and Heyting algebras. In Section \ref{s3} we generalize
Kalman's functor for the category whose objects are posets with
first element and whose morphisms are maps which preserve finite
existing infima and the first element (note that the morphisms of
this category are in particular order-preserving maps). Moreover,
we apply the mentioned equivalence in order to build up an
equivalence for the category whose objects are bounded
semilattices and whose morphisms are the corresponding algebra
homomorphisms. In Section \ref{s4} we recall definitions and
properties about hemi-implicative semilattices (lattices)
\cite{SM2}, Hilbert algebras with infimum \cite{Fig}, implicative
semilattices \cite{N} and semi-Heyting algebras \cite{Sanka}. In
Section \ref{s5} we employ results of sections \ref{s3} and
\ref{s4} in order to establish equivalences, following the
original Kalman's construction, for the categories of bounded
hemi-implicative semilattices, bounded Hilbert algebras with
infimum, bounded implicative semilattices, hemi-implicative
lattices, respectively, and the category of semi-Heyting algebras.
Finally, in Section \ref{s6} we introduce and study the notion of
well-behaved congruences for the objects corresponding to the
categories introduced in Section \ref{s5}. \vspace{1pt}

We give a table with some of the categories we shall consider in
this paper: \vspace{10pt}

\small

\begin{tabular}{c|c| c c}

$\mathbf{Category}$ &       $\mathbf{Objects}$                                                     & $\mathbf{Morphisms}$ \\
\hline
&&\\
$\DL$                       &Bounded distributive lattices                                         & Algebra homomorphisms   \\
$\cKl$                      &Centered Kleene algebras                                              & Algebra homomorphisms   \\
$\He$                       &Heyting algebras                                                      & Algebra homomorphisms       \\
$\cN$                       &Centered Nelson algebras                                              & Algebra homomorphisms\\
$\cNL$                      &Centered Nelson lattices                                              & Algebra homomorphisms\\
$\Po$                       &Posets with bottom                                                    & Certain order morphisms \\
$\KPo$                      &Kleene posets                                                         & Certain order morphisms \\
$\MS$                       &Bounded semilattices                                                  & Algebra homomorphisms \\
$\KMS$                      &Certain objects of $\KPo$                                             & Morphisms of $\KPo$ \\
$\hIS$                      &Bounded hemi-implicative semilattices                                  & Algebra homomorphisms \\
$\hBDL$                     &Hemi-implicative lattices                                              & Algebra homomorphisms  \\

\end{tabular}

\begin{tabular}{c|c| c c}
$\mathbf{Category}$ &   $\mathbf{Objects}$         & $\mathbf{Morphisms}$ \\
\hline
$\Hil$                      &Bounded Hilbert algebras with infimum                                 & Algebra homomorphisms \\
$\IS$                       &Bounded implicative semilattices                                      & Algebra homomorphisms \\
$\SH$                       &Semi-Heyting algebras                                                 & Algebra homomorphisms \\
$\KhIS$                     &Objects of $\KMS$ with an additional                                  & Certain morphisms of $\KMS$ \\
       & operation & \\
$\KHil$                      &Certain objects of $\KhIS$                                            & Morphisms of $\KhIS$ \\
$\KIS$                       &Certain objects of $\KhIS$                                            & Morphisms of $\KhIS$ \\
$\KhBDL$                     &Objects of $\cKl$ with an additional                                  & Certain morphisms of $\cKl$ \\
 & operation & \\
$\KSH$                       &Certain objects of $\KhBDL$                                           & Morphisms of $\KhBDL$ \\

\end{tabular}

\normalsize

\vspace{10pt}

If $\mathrm{A}$ is one of the categories $\cKl$, $\KPo$, $\KMS$,
$\KhIS$, $\KHil$, and $\KhBDL$, then we write $\mathrm{A}^{\CK}$
to denote the full subcategory of $\mathrm{A}$ whose objects
satisfy the condition $(\CK)$, that will be defined later.

The results we expound in the present paper are motivated by the
abstraction of ideas coming from different varieties of algebras
related to some constructive logics, as Heyting algebras and
Nelson algebras, and in particular by the existent categorical
equivalence between the category of Heyting algebras and the
category of centered Nelson algebras (see \cite{cig}) combined
with the fact that the variety of centered Nelson algebras is term
equivalent to the variety of centered Nelson lattices, as it is shown in
\cite{SV} (see also \cite{B}). In this paper we introduce and
study categories which are closely connected with the category of
centered Nelson lattices, as for instance the category $\KPo$ of
Kleene posets (of which centered Nelson lattices can be seen as
particular cases) and the category $\KhBDL$ of centered Kleene
algebras endowed with a binary operation which generalizes the
implication of Nelson lattices. We consider that the study of the
above mentioned categories is interesting in itself. We also think
that the categorical equivalences and some related properties
studied in this paper can be of interest for future work
concerning the understanding of the categories of bounded
hemi-implicative semilattices and hemi-implicative lattices,
respectively.

\section{Basic results}\label{Kc}

The definition of the functor from the category of Kleene algebras
to the category of bounded distributive lattices given by R.
Cignoli \cite{cig} is based on Priestley duality, and the
interpolation property for Kleene algebras considered by Cignoli
in establishing the equivalence is stated in topological terms. On
the other hand, M. Sagastume proved in an unpublished manuscript
\cite{Sf} that in centered Kleene algebras the interpolation
property is equivalent to an algebraic condition called (CK), that
we will state later on. Moreover, she presented an equivalence
between the category of bounded distributive lattices and the
category of centered Kleene algebras that satisfy (CK), but using
a different (purely algebraic) construction to that given by R.
Cignoli in \cite{cig}. In what follows we describe this
equivalence whose details can be found in \cite{CCSM}.

We assume the reader is familiar with bounded distributive
lattices and Heyting algebras \cite{BD}. A \emph{De Morgan}
algebra is an algebra $(H,\we,\vee,{\sim},0,1)$ of type
$(2,2,1,0,0)$ such that $(H,\we, \vee,0,1)$ is a bounded
distributive lattice and ${\sim}$ fulfills the equations
\[
\text{${\sim} {\sim}
x = x$ \;\; and \;\; ${\sim}(x \vee y) = {\sim} x \we {\sim} y$.}
\]
An operation ${\sim}$ which satisfies the previous two equations
is called \emph{De Morgan involution}. A \emph{Kleene algebra} is
a De Morgan algebra in which the inequality
\[
x\we {\sim} x \leq y \vee {\sim} y
\] holds. A
\emph{centered Kleene algebra} is an algebra
$(H,\we,\vee,{\sim},c, 0,1)$ where the algebra $(H,\we,\vee,$
${\sim},0,1)$ is a Kleene algebra and $\ce$ is an element such
that $\ce = {\sim} \ce.$ It is immediate to see that $\ce$ is
necessarily unique. The element $\ce$ is called \emph{center}. We
write $\DL$ for the category of bounded distributive lattices and
$\cKl$ for the category of centered Kleene algebras. In both cases
the morphisms are the corresponding algebra homomorphisms. It is
interesting to note that if $T$ and $U$ are centered Kleene algebras
and $f:T\ra U$ is a morphism of Kleene algebras then $f$ preserves
necessarily the center, i.e., $f(\ce) = \ce$.

The functor $\K$ from the category $\DL$ to the category $\cKl$ is
defined as follows. For an object $H\in \DL$ we let
\[ \K(H): =\{(a,b) \in H\times H: a\we b = 0\}.
\]
This set is endowed with the operations and the distinguished
elements defined by:
\begin{eqnarray*}
   (a,b)\vee (d,e) & := & (a\vee d,b\we e)\\
   (a,b)\we (d,e)& := & (a\we d,b\vee e)\\
   {\sim} (a,b)& := & (b,a)\\
0 & := & (0,1)\\
1 & := & (1,0)\\
\ce & := & (0,0)
\end{eqnarray*}
We have that $(\K(H),\we,\vee,\sim,\ce,0,1)\in \cKl$.

For a morphism  $f:H \ra G \in \DL$, the map $\K(f):\K(H) \ra
\K(G)$ defined by
$$\K(f)(a,b) = (f(a),f(b))$$
is a morphism in $\cKl$. Hence, $\K$ is a functor from $\DL$ to
$\cKl$.

Let $(T,\we,\vee,{\sim},\ce,0,1)\in \cKl$. The set
$$\C(T):=\{x\in T:x\geq \ce\}$$ is the universe of a subalgebra of
$(T,\we,\vee,\ce,1)$ and $(\C(T),\we,\vee,\ce,1) \in \DL$.
Moreover, if $g:T\ra U$ is a morphism in $\cKl$, then the map
$\C(g): \C(T) \ra \C(U)$, given by $\C(g)(x) = g(x)$, is a
morphism in $\DL$. Thus, $\C$ is a functor from $\cKl$ to $\DL$.

Let $H \in \DL$.  The map $\alpha_{H}: H \ra \C(\K(H))$ given by
$\alpha_{H}(a) = (a,0)$ is an isomorphism in $\DL$. If $T \in
\cKl$, then the map $\beta_T: T\ra \K(\C(T))$ given by $\beta_T(x)
= (x\vee \ce, {\sim} x \vee \ce)$ is injective and a morphism in
$\cKl$. It is not difficult to show that the functor $\K: \DL \ra
\cKl$ has as left adjoint the functor $\C: \cKl \ra \DL$ with unit
$\beta$ and counit $\alpha^{-1}$.

We are interested though in an equivalence between $\DL$ and the
full subcategory of $\cKl$ whose objects satisfy the condition
(\ref{eq-CK}) we proceed to state.

Let $T\in \cKl$. We consider the  algebraic condition:
\begin{equation} \label{eq-CK}
(\forall x, y \geq c)(x\we y = \ce \
\longrightarrow \ (\exists z)(z\vee \ce = x \ \& \ {\sim} z \vee
\ce = y)). \tag{$\CK$}
\end{equation}
This condition characterizes the surjectivity of $\beta_T$, that
is, for every $T\in \cKl$, $T$ satisfies (\ref{eq-CK}) if and only
if $\beta_T$ is a surjective map, as shown in \cite{Sf}. The
condition (\ref{eq-CK}) is not necessarily verified in every
centered Kleene algebra (see \cite{CCSM}).

We write $\cKlp$ for the full subcategory of $\cKl$ whose objects
satisfy  (\ref{eq-CK}). The functor $\K$ can then  be seen as a
functor from $\DL$ to $\cKlp$. The next theorem was proved by M.
Sagastume in \cite{Sf}. A complete proof of it can be also found
in \cite{CCSM}.

\begin{thm} \label{Kce}
The functors $\K$ and $\C$ establish a categorical equivalence
between $\DL$ and $\cKlp$ with natural isomorphisms  $\alpha$ and
$\beta$.
\end{thm}

Let $T\in \cKl$. We know that $\beta_T$ is not necessarily a
surjective map. However we will prove that $\beta_T$ is an
epimorphism. Before, we need a lemma that is interesting in its
own right. It tells us that the morphisms in $\cKl$ are determined
by their behavior on the elements greater than or equal to the
center.

\begin{lem}
\label{lem:morph-C(T)} If $f:T\ra U$ and $g:T\ra U$ are morphisms
in $\cKl$ and $f(x) = g(x)$ whenever $x\in \C(T)$, then $f(x) =
g(x)$ for every $x\in T$.
\end{lem}
\begin{proof}
Suppose that $f(x) = g(x)$ whenever $x\in \C(T)$. Let $x$ be an
arbitrary element of $T$. Then
\[
\begin{array}
[c]{lllll}
f(x) \vee \ce & = &  f(x) \vee f(\ce) &  & \\
& = & f(x\vee \ce)&  & \\
& = & g(x\vee \ce)&  & \\
& = & g(x) \vee g(\ce)&  & \\
& = & g(x) \vee \ce,& &
\end{array}
\]
so we obtain that $f(x) \vee \ce = g(x) \vee \ce$. Similarly we
can prove that ${\sim} f(x) \vee \ce = {\sim} g(x) \vee \ce$,
which is equivalent to $f(x) \we \ce = g(x) \we \ce$. Hence, it
follows from the distributivity of the underlying lattice of $U$
that $f(x) = g(x)$.
\end{proof}

\begin{prop} \label{beta}
Let $T\in \cKl$. Then $\beta_T$ is an epimorphism.
\end{prop}

\begin{proof}
Let $f: \K(\C(T)) \ra U$ and $g:\K(\C(T))\ra U$ be morphisms in
$\cKl$ such that $f \circ \beta_T = g \circ \beta_T$, where
$\circ$ denotes the composition of functions. We will prove that
$f = g$. Let $(x,y) \in \C(\K(\C(T)))$, i.e., $x\we y = \ce$,
$x\geq \ce$, $y \geq \ce$ and $(\ce,\ce) \leq (x,y)$, where we
also write $\leq$ for the order associated to the underlying
lattice of $\K(\C(T))$. In particular we have that $y\leq \ce$, so
$y = \ce$. Then $(x,y) = (x,\ce)$. Besides, since $x\geq \ce$ we
have that $\beta_{T}(x) = (x,\ce)$. Then
\[
\begin{array}
[c]{lllll}
f(x,y) & = & f(x,\ce) &  & \\
& = & (f \circ \beta_{T})(x)&  & \\
& = & (g \circ \beta_{T})(x)&  & \\
& = & g(x, \ce)&  & \\
& = & g(x,y).& &
\end{array}
\]
Hence, $f(x,y) = g(x,y)$ whenever $(x,y) \in \C(\K(\C(T)))$.
Therefore, it follows from Lemma \ref{lem:morph-C(T)} that $f(x,y)
= g(x,y)$ for every $(x,y) \in \K(\C(T))$, which was our aim.
\end{proof}

Let $H \in \DL$ and $a$, $b\in H$. If the relative
pseudocomplement of $a$ with respect to $b$ exists, then we denote
it by $a \ra_{\He} b$. Recall that a \emph{Nelson algebra}
\cite{cig} is a Kleene algebra such that for each pair $x$, $y$
there exists the binary operation $\Ra$ given by $x\Ra y: = x
\ra_{\He} ({\sim x} \vee y)$ and for every $x,y,z$ it holds that
$(x \we y)\Rightarrow z = x \Rightarrow (y\Rightarrow z)$. The
binary operation $\Rightarrow$ so defined is called the weak
implication.

We denote by $\He$ the category of Heyting algebras. M. Fidel
\cite{F} and D. Vakarelov \cite{V} proved independently that if
$H\in \He$, then the Kleene algebra $\K(H)$ is a Nelson algebra,
in which the weak implication is defined  for pairs
$(a, b)$ and $(d, e$) in $\K(H)$ as follows:
\begin{equation} \label{ic}
(a,b)\Rightarrow(d,e):= (a\ra d, a\we e).
\end{equation}
We say that an algebra $(T,\we,\vee,\Rightarrow,{\sim},\ce,0,1)$
is a \emph{centered Nelson algebra} if the reduct
$(T,\we,\vee,\Rightarrow,{\sim},0,1)$ is a Nelson algebra and
$\ce$ satisfies ${\sim} \ce = \ce$. We write $\cN$ for the
category of centered Nelson algebras.

The following result appears in \cite[Proposition 3.7]{camesa2}
and  is a reformulation of \cite[Theorem 3.14]{cig}.

\begin{thm} \label{adjHey}
The functors $\K$ and $\C$ establish a categorical equivalence
between $\He$ and $\cN$ with natural isomorphisms $\alpha$ and
$\beta$.
\end{thm}

We assume the reader is familiar with commutative residuated
lattices \cite{Ts}. An \emph{involutive residuated lattice} is a
bounded, integral and commutative residuated lattice $(T,\we,
\vee, \ast,\ra, 0, 1)$ such that for every $x\in T$ it holds that
$\neg \neg x = x$, where $\neg x: = x\ra 0$ and $0$ is the first
element of $T$ \cite{B}. In an involutive residuated lattice it
holds that $x \ast y = \neg (x \ra \neg y)$ and $x\ra y = \neg (x
\ast \neg y)$. A \emph{Nelson lattice} \cite{B} is an involutive
residuated lattice $(T,\we,\vee, *,\ra,0,1)$ which satisfies the
additional inequality $(x^2 \ra y)\we ((\neg y)^2 \ra \neg x) \leq
x\ra y$, where $x^2:=x\ast x$. See also \cite{V}.

\begin{rem} \label{br3}
Let $(T,\we, \vee, \Rightarrow,{\sim}, 0,1)$ be a Nelson algebra.
We define on $T$ the binary operations $*$ and $\ra$ by
\begin{eqnarray*}
x*y:=& {\sim} (x \Rightarrow {\sim} y) \vee {\sim} (y \Rightarrow
{\sim} x), \hspace{1cm} x \ra y :=&  (x \Rightarrow y) \we ({\sim}
y\Rightarrow {\sim} x).
\end{eqnarray*}
Then Theorem 3.1 of \cite{B} says that $(T,\we, \vee, \ra,*, 0,1)$
is a Nelson lattice. Moreover, ${\sim} x = \neg x = x\ra 0$.

Let $(T,\we,\vee,*,\ra,0,1)$ be a Nelson lattice. We define on $T$
a binary operation $\Rightarrow$ and a unary operation $\sim$ by
\begin{eqnarray*}
   x \Rightarrow y:=& x^2 \ra y, \hspace{1cm}
   {\sim} x:=& \neg x,
\end{eqnarray*}
where $x^2 = x*x$. Then Theorem 3.6 of \cite{B} says that the
algebra $(T,\we, \vee,\Rightarrow,{\sim},0,1)$ is a Nelson
algebra.

In \cite[Theorem 3.11]{B} it was also proved that the category of
Nelson algebras and the category of Nelson lattices are
isomorphic. Taking into account the construction of this
isomorphism in \cite{B} we have that the variety of Nelson
algebras and the variety of Nelson lattices are term equivalent
and the term equivalence is given by the operations we have
defined before.
\end{rem}

The results from \cite{B}  about the connections between Nelson
algebras and Nelson lattices mentioned in Remark \ref{br3} are
based on results from Spinks and Veroff \cite{SV}. In particular,
the term equivalence of the varieties of Nelson algebras and
Nelson lattices  was discovered by Spinks and Veroff in \cite{SV}.
\vspace{1pt}

A \emph{centered Nelson lattice} is an algebra
$(T,\we,\vee,*,\ra,\ce,0,1)$, where the reduct
$(T,\we,\vee,*,\ra,0,1)$ is a Nelson lattice and $\ce$ is an
element such that $\neg \ce = \ce$. It follows from Remark
\ref{br3} that the variety of centered Nelson algebras and the
variety of centered Nelson lattices are term equivalent. We write
$\cNL$ for the category of centered Nelson lattices.

\begin{rem} \label{imp}
Let $(H,\we, \vee,\ra,0,1) \in \He$. Then
$(\K(H),\we,\vee,\Ra,{\sim},\ce,0,1) \in \cN$. Hence it follows
from Remark \ref{br3} that $(\K(H),\we,\vee,*,\ra,\ce,0,1) \in
\cNL$, where for $(a,b)$ and $(d,e)$ in $\K(H)$ the operations
$\ast$ and $\ra$ take the form
\begin{eqnarray*}
   (a,b) * (d,e) =  (a\we d, (a\ra e)\we (d\ra b)),\\
   (a,b) \ra (d,e) =  ((a\ra d)\we (e\ra b), a \we e).
\end{eqnarray*}
We write $\ra$ both for the implication in $H$ as for the
implication in $\K(H)$.
\end{rem}

It follows from Theorem \ref{adjHey} and Remark \ref{br3} that
there is a categorical equivalence between $\He$ and $\cNL$, as it
was also mentioned in \cite[Corollary 2.11]{CCSM}. In what follows
we will make explicit a construction of this equivalence.

\begin{prop} \label{Nl}
The functors $\K$ and $\C$ establish a categorical equivalence
between $\He$ and $\cNL$ with natural isomorphisms $\alpha$ and
$\beta$.
\end{prop}

\begin{proof}
Let $H \in \He$. Then the centered Kleene algebra
$(\K(H),\we,\vee,{\sim},\ce,0,1)$ endowed with the two operations
given in Remark \ref{imp} is a centered Nelson lattice. It is
immediate that if $f$ is a morphism in $\He$,  then $\K(f)$ is a
morphism in $\cNL$.

Let $(T,\we,\vee,*,\ra, \ce,0,1)\in \cNL$. Taking into account
Remark \ref{br3} we deduce that $(T,\we, \vee,\Ra, {\sim},
\ce,0,1)\in \cN$, where $x \Ra y = x^{2} \ra y$. Moreover,
\begin{equation} \label{nanl1}
x \ra y = (x \Ra y) \we ({\sim} y \Ra {\sim x}).
\end{equation}
Let $x$, $y\geq \ce$. We will prove that $x \ra y = x \ra_{\He}
y.$ In order to show it, note that straightforward computations
show that
\begin{equation} \label{nanl2}
x\Ra y = x \ra_{\He} y.
\end{equation}
Besides, ${\sim} y \Rightarrow {\sim} x = {\sim} y \ra_{\He} (y \vee
{\sim} x)$. Since $y\geq \ce$ and ${\sim} x \leq \ce$, then $y \vee
{\sim} x = y$. Then ${\sim} y \Rightarrow {\sim} x = {\sim} y \ra_{\He}
y$. Hence, it follows from (\ref{nanl1}) and (\ref{nanl2}) that
$$ x\ra y = (x\ra_{\He} y) \we ({\sim} y \ra_{\He} y).$$
Note that $x\ra y = x\ra_{\He} y$ if and only if $x\ra_{\He} y
\leq {\sim} y \ra y$, which is equivalent to ${\sim} y \we
(x\ra_{\He} y) \leq y$. But ${\sim} y \leq \ce$ and $x \ra_{\He} y
\geq \ce$, so ${\sim} y \we (x\ra_{\He} y) = {\sim} y$. Hence,
${\sim} y \we (x\ra_{\He} y) \leq y$ if and only if ${\sim} y \leq
y$. Since $y\geq \ce$, then ${\sim} y \leq \ce$, so ${\sim} y \leq
y$. Then we have that $x\ra y = x \ra_{\He} y$. Thus, $(\C(T),
\we, \vee,\ra,\ce,1) \in \He$. Straightforward computations show
that if $g$ is a morphism in $\cNL$, then $\C(g)$ is a morphism in
$\He$.

It is also immediate that if $H\in \He$ then $\alpha_H$ is an
isomorphism in $\He$. Let $(T,\we, \vee,*,\ra, \ce,0,1)\in \cNL$.
It follows from Theorem \ref{adjHey} and Remark \ref{br3} that
$\beta_T$ preserves $\ra$. Therefore, $\beta_T$ is an isomorphism
in $\cNL$.
\end{proof}

The main goal of this paper is to find a generalization of
Proposition \ref{Nl} replacing the categories of Heyting algebras
and centered Nelson lattices by the categories of bounded
hemi-implicative semilattices and hemi-implicative lattices,
respectively. To make it possible, we start studying an
equivalence for a particular category of posets with first
element. Then we  employ it  to obtain an equivalence for the
category of bounded semilattices. Finally, taking into account the
last mentioned equivalence, we build up an equivalence for the
categories of bounded hemi-implicative semilattices and
hemi-implicative lattices, respectively, and for some of its full
subcategories.

\section{Kalman's functor for posets with bottom and for bounded
semilattices} \label{s3}

In this section, we generalize the equivalence given for the
category of bounded distributive lattices but replacing this
category by the category whose objects are posets with first
element and whose morphisms are maps which preserve finite
existing infima and the first element. Then we apply this
equivalence in order to establish an equivalence for the category
whose objects are bounded semilattices and whose morphisms are the
corresponding algebra homomorphisms. We start with some
preliminary definitions and properties. \vspace{1pt}

Let $(P,\leq,0)$ be a poset with first element, and let $(P \times
P,\preceq)$ be the poset with universe the cartesian product $P\times P$
where the order $\preceq$ is given by
\[
(a,b) \preceq (d,e)\; \text{if and only if}\; a\leq
d\;\text{and}\; e \leq b.
\]
In other words, $(P \times P,\preceq)$ is the direct product of
$(P,\leq)$ with its dual. Let $(P,\leq)$ and $(Q,\leq)$ be posets.
Let $f:(P,\leq) \ra (Q,\leq)$ be a function. We say that $f$
\emph{preserves finite existing infima} if for every $a$, $b\in P$
such that $a\we b$ exists in $P$ then $f(a)\we f(b)$ exists in $Q$
and $f(a\we b) = f(a)\we f(b)$.

\begin{defn}
The category $\Po$ has as objects the posets with first element
and has as morphisms the maps between posets with first element
which preserve  the finite existing infima and the first element.
\end{defn}
Note that every morphism in $\Po$ preserves the order. It follows
from the fact that morphisms preserve the finite existing infima.

Let $P \in \Po$. We define the following set:
\begin{equation}
\K(P): = \{(a,b)\in P\times P:  a\we
b\;\text{exists and}\; a\we b = 0\}.
\end{equation}
This set is the natural one to associate with the poset $P$ if we
aim to generalize the original Kalman's construction given for
bounded distributive lattices \cite{K}. To attain the
generalization we first order $\K(P)$ with the order induced by
the poset $(P \times P,\preceq)$ defined above. It is immediate
from the definition that  $\K(P)$ is closed under the unary
operation ${\sim}$ on $P \times P$ given by ${\sim}(a, b) = (b,
a)$, and that the element $\ce = (0, 0)$ belongs to $\K(P)$. Thus
we obtain the structure $\K(P) := (\K(P),\preceq,{\sim}, \ce).$

The following elemental lemma plays a fundamental role in
some proofs of this section.

\begin{lem} \label{l1}
Let  $(b,d) \in \K(P)$. The following conditions hold:
\begin{enumerate}[\normalfont 1.]
\item  For every $a \in P$, $(a,0) \we (b,d)$ exists in $\K(P)$ if
and only if $a\we b$ exists in $P$. If these conditions hold, then
$(a,0) \we (b, d) = (a\we b, d)$.
\item $(b,d) \we (0,0)$ exists  in  $\K(P)$ and $(b,d) \we (0,0) =
(0,d)$. \item $(b,d) \vee (0,0)$ exists in  $\K(P)$  and $(b,d)
\vee (0,0) = (b,0).$
\end{enumerate}
\end{lem}

\begin{proof}   In general, if  $y = (e, u) \in  \K(P)$,
we write $\pi_1(y)$ for the first coordinate and $\pi_2(y)$ for
the second coordinate (i.e., $\pi_1(y) = e$ and  $\pi_2(y) = u$).
Let $(b,d) \in \K(P)$. We proceed to the proof.

1. Suppose that $a \in P$ and $(a,0) \we (b,d)$ exists in $\K(P)$.
Let  $x := (a,0) \we (b,d)$. We have that $x \preceq (a,0)$ and
$x\preceq (b,d)$, so by the definition of $\preceq$ we obtain that
$\pi_1(x) \leq a$ and $\pi_1(x) \leq b$. Thus $\pi_1(x)$ is a
lower bound of the set $\{a,b\}$. Let now $e$ be a lower bound of
$\{a,b\}$, i.e., $e\leq a$ and $e\leq b$. The fact that $(b,d) \in
\K(P)$ implies that $(e,d) \in \K(P)$. Since $(e,d)\preceq (a,0)$
and $(e,d) \preceq (b,0)$,  then $(e,d)\preceq x$. Hence, $e\leq
\pi_1(x)$. Thus, $\pi_1(x) = a\we b$. Conversely, suppose that
$a\we b$ exists in $P$.  We have that $(a\we b,d) \in \K(P)$,
$(a\we b,d) \preceq (a,0)$,  and $(a\we b,d) \preceq (b,d)$. Let
$(e, u) \in \K(P)$ such that $(e,u)\preceq (a,0)$ and $(e,u)
\preceq (b,d)$, i.e., $e\leq a$, $e\leq b$ and $d\leq u$. Since $e
\leq a\we b$, then $(e,u) \preceq (a\we b, d)$. Hence we obtain
that  $(a,0)\we (b,d)$ exists in $\K(P)$,  and moreover $(a,0)\we
(b,d) = (a\we b,d)$.

2. To prove that $(b,d) \we (0,0)$ exists in  $\K(P)$, let us see
that $(0,d)$ is the infimum of $(b,d)$ and $(0,0)$. We have that
$(0,d)\preceq (b,d)$ and $(0,d)\preceq (0,0)$, so $(0,d)$ is a
lower bound of $\{(b,d), (0,0)\}$. Let $(e,u) \in \K(P)$ be a
lower bound of $\{(b,d), (0,0)\}$, so $(e,u) \preceq (b,d)$ and
$(e,u) \preceq (0,0)$. Hence, $e = 0$ and $d\leq u$, and so $(e,u)
\leq (0,d)$. Therefore we obtain that $(b,d) \we (0,0) = (0,d)$.

3. In a similar way it can be proved that $(b,d) \vee (0,0)$
exists in $\K(P)$ and is $(b,0)$.
\end{proof}

Motivated by properties of $\K(P)$ we give the following
definition.

\begin{defn}
\label{def:Kleeneposet}
A structure $(T,\leq,{\sim},\ce)$ is a \textit{Kleene poset} if
the following conditions hold:
\begin{enumerate}[\normalfont 1.]
\item $(T,\leq)$ is a poset. \item ${\sim}$ is an unary operation
on $T$ which is an involution, i.e., ${\sim} {\sim} x = x$ for
every $x\in T$ and is order reversing, i.e., for every  $x,y\in
T$, if $x\leq y$, then ${\sim} y \leq {\sim} x$. \item  $\ce =
{\sim} \ce$. \item  $x\vee \ce$ exists, for every $x \in T$. \item
For every $x\in T$, $(x\vee \ce) \we ({\sim} x \vee \ce)$ exists
and $(x\vee \ce) \we ({\sim} x \vee \ce) = \ce$. \item For every
$x,y \in T$, if $x\we \ce \leq y \we \ce$ and $x\vee \ce \leq y
\vee \ce$, then $x \leq y$.
\end{enumerate}
\end{defn}

The element $\ce$ of the previous definition will be also called
\emph{center}. The next lemma justifies the use of $\we$ in the
statement of the condition 6.

\begin{lem} \label{la}
Let $(T, \leq)$ be a poset satisfying $2$., $3$., $4$. and $5$.\ of
Definition \ref{def:Kleeneposet}.
\begin{enumerate}[\normalfont 1.]
\item Let $x,y \in T$. If  $x\we y$ exists, then ${\sim} x \vee
{\sim} y$ exists and ${\sim} x \vee {\sim} y = {\sim} (x\we y)$.
Analogously, if  $x\vee y$ exists,  then  ${\sim} x \wedge {\sim}
y$ exists and  ${\sim} x \wedge {\sim} y = {\sim} (x\vee y)$.
\item For every $x\in T$, $x\we \ce$ exists and  $x\we \ce =
{\sim} ({\sim} x \vee \ce)$. \item The element $\ce$ is unique.
\end{enumerate}
\end{lem}

\begin{proof}
Straightforward computations show the first two assertions. In
order to prove that the center is unique, let $\ce$ and $\ce'$ be
centers. Then $\ce = (\ce' \vee \ce)\we ({\sim} \ce' \vee \ce)$.
Since ${\sim} \ce' = \ce'$, then $\ce = \ce' \vee \ce$. Hence,
$\ce'\leq \ce$. Analogously, $\ce\leq \ce'$. Thus, $\ce = \ce'$.
\end{proof}

In what follows we introduce the category of Kleene posets.

\begin{defn}
We denote by $\KPo$ the category whose objects are the Kleene
posets and whose morphisms are the maps $g$ between Kleene posets
that preserve the order, the involution and the finite existing
infima over elements greater than or equal to the center.
\end{defn}

Note that if $g:T\ra U$ is a morphism in $\KPo$,  $\ce$ is the
center of $T$ and $\ce'$ is the center of $U$, then $g(\ce) =
\ce'$. It follows from the fact that $g(\ce) = {\sim} g(\ce)$ and
the fact that the center is unique.

If $T\in \KPo$,  we define $\C(T)$ as in the case of centered
Kleene algebras. If $f$ is a morphism in $\Po$ and $g$ is a
morphism in $\KPo$ we define $\K(f)$ and $\C(g)$ as in Section
\ref{int}, respectively.

\begin{lem}
\begin{enumerate} [\normalfont (a)]
\item If $(P,\leq,0) \in \Po$, then $(\K(P),\preceq,{\sim}, \ce) \in
\KPo$. \item If $f \in \Po$, then $\K(f) \in \KPo$.
\end{enumerate}
\end{lem}

\begin{proof}
It follows from straightforward computations based on Lemma
\ref{l1} that if $(P,\leq,0) \in \Po$, then
$(\K(P),\preceq,{\sim},$ $ \ce) \in \KPo$. In what follows we will
prove that if $f:P \ra Q$ is a morphism in $\Po$, then
$\K(f):\K(P)\ra \K(Q)$ is a morphism in $\KPo$.

By the definition of $\Po$ we have that if $(a,b) \in \K(P)$, then
$f(a) \wedge f(b)$ exists in $Q$ and is $f(a\we b) = f(0)$, which
is $0$ in $Q$. Thus, $(f(a),f(b))\in \K(Q)$. Therefore, $\K(f)$ is
indeed a map from $\K(P)$ to $\K(Q)$.

Since $f$ preserves the order, then $\K(f)$ preserves the order. It
is immediate that $\K(f)$ preserves the involution.

Let $(a,0)$ and $(b,0)$ be elements such that $(a,0) \we (b,0)$
exists. It follows from Lemma \ref{l1} that  $a\we b$ exists and
$(a,0) \we (b,0) = (a\we b,0)$. Then, $f(a)\we f(b)$ exists and
$f(a\we b) = f(a)\we f(b)$. Again by Lemma \ref{l1} we obtain that
$(f(a),0)\we (f(b),0)$ exists and $(f(a),0)\we (f(b),0) = (f(a)\we
f(b),0)$. Hence, $\K(f)((a,0)\we (b,0)) = \K(f)(a,0) \we
\K(f)(b,0)$.

Therefore, $\K(f)$ is a morphism in $\KPo$.
\end{proof}

Using the previous lemma, it is immediate to see that $\K$ defines
a functor from $\Po$ to $\KPo$.

Let $P \in \Po$. The map $\alpha_P:P \ra \C(\K(P))$ given by
$\alpha_{P}(a,b) = (a,0)$ is easily seen to be an isomorphism in
$\Po$. The fact that $\alpha_P$ is morphism is a consequence of
Lemma \ref{l1}.

The proof of the following lemma is immediate. It easily follows
from it  that $\C: \KPo \ra \Po$ is a functor.

\begin{lem}
\begin{enumerate}[\normalfont (a)]
\item[]
\item If $(T, \leq, {\sim},\ce) \in \KPo$ then $(\C(T),\leq,\ce) \in
\Po$. \item If $g \in \KPo$ then $\C(g) \in \Po$.
\end{enumerate}
\end{lem}

\begin{defn}\label{DCK}
For $T\in \KPo$ we also name $(\CK)$
 to the following condition
\begin{equation}
\label{eq-CK-2}
(\forall x, y \geq c)(\text{if } x\we y \text{
exists and } x\we y = \ce, \text{ then } (\exists z)(z\vee \ce = x
\ \& \ {\sim} z \vee \ce = y)). \tag{$\CK$}
\end{equation}
\end{defn}

\begin{rem}
In Section \ref{Kc}, the condition $(\CK)$ was defined for
centered Kleene algebras. Notice that if
$(T,\we,\vee,{\sim},\ce,0,1)$ is a centered Kleene algebra,  then
the structure $(T,\leq,{\sim},\ce)$ is a Kleene poset, where
$\leq$ is the order associated with the lattice $(T,\we,\vee)$. In
particular, we have that $(T,\we,\vee,{\sim},\ce,0,1)$ satisfies
the condition $(\CK)$ given in Section \ref{Kc} if and only if
$(T,\leq,{\sim},\ce)$ satisfies the condition $(\CK)$ given in
Definition \ref{DCK}. This fact justifies the use of the same
label for both conditions.
\end{rem}

As in the case of bounded distributive lattices, if $(P,\leq,0)
\in \Po$, then the structure $(\K(P),\preceq,{\sim}, \ce)$
satisfies $(\CK)$.

\begin{rem}\label{r0}
For $T\in \KPo$ and $x\in T$ we have $(x\vee \ce) \we ({\sim} x
\vee \ce) = \ce$, which shows that the map $\beta_T:T \ra
\K(\C(T))$ defined by $\beta_{T}(x) = (x\vee \ce,{\sim} x \vee
\ce)$ is a well defined map. We also have that $T$ satisfies
$(\CK)$ if and only if $\beta_T$ is surjective.
\end{rem}

Let $f:(P,\leq) \ra (Q,\leq)$ be an order isomorphism, i.e., a
bijective map such that for every $a$, $b\in P$, $a\leq b$ if and
only if $f(a)\leq f(b)$. Let $a$, $b\in P$ such that $a\we b$
exists. Then $f(a) \we f(b)$ exists and $f(a) \we f(b) = f(a\we
b)$. Straightforward computations prove the following remark.

\begin{rem} \label{r1}
Let $f:T\ra U$ be a morphism in $\KPo$. If $f$ is an order
isomorphism, then $f$ is an isomorphism in $\KPo$.
\end{rem}

\begin{lem}
If $T \in \KPo$,  then $\beta_T$ is an injective morphism in $\KPo$.
Moreover, if $T$ satisfies $(\CK)$, then $\beta_T$ is an
isomorphism in $\KPo$.
\end{lem}

\begin{proof}
In order to show that $\beta_T$ preserves the order, let $x,y \in
T$ such that $x\leq y$. Then $x\vee \ce \leq y \vee \ce$ and
${\sim} x \vee \ce \geq {\sim} y \vee \ce$, which means that
$\beta_T(x) \leq \beta_T(y)$. Thus, $\beta_T$ preserves the order.
It is immediate that $\beta_T$ preserves the involution. Let now
$x,y \in T$ such that $x,y \geq \ce$. Assume that  $x\we y$
exists. So $\beta_T(x\we y) = (x\we y, \ce)$. Moreover, we have
that $\beta_T(x) = (x,\ce)$ and $\beta_T(y) = (y,\ce)$. Thus, it
follows from Lemma \ref{l1} that  $(x,c) \we (y,c)$ exists and
$(x,\ce) \we (y,\ce) = (x\we y, \ce)$. Then $\beta_T(x \we y) =
\beta_T(x) \we \beta_T(y)$. Hence, $\beta_T$ is a morphism in
$\KPo$.

Now we will prove that for every $x,y \in T$, $x\leq y$ if and
only if $\beta_T(x) \leq \beta_T(y)$. The fact that if $x\leq y$,
then $\beta_T(x) \leq \beta_T(y)$ was proved before. In order to
prove the converse, suppose that $\beta_T(x) \leq \beta_T(y)$,
i.e., $x\vee \ce \leq y \vee \ce$ and $x \we \ce \leq y \we \ce$.
So, by the definition of Kleene poset  we have $x\leq y$. In particular,
$\beta_T$ is an injective map.

Finally, assume that $T$ satisfies $(\CK)$. It follows from
remarks \ref{r0} and \ref{r1} that $\beta_T$ is an isomorphism in
$\KPo$.
\end{proof}

Straightforward calculations prove that if $f:P\ra Q$ is a
morphism in $\Po$ then $ (\C \circ \K)(f) \circ \alpha_{P} =
\alpha_{Q} \circ f$, and if $g:T \ra U$ is a morphism in $\KPo$
then $ (\K \circ \C)(g) \circ \beta_{T} = \beta_U \circ g$.

\begin{thm} \label{Kfp}
Let $\KPop$ be the full subcategory of $\KPo$ whose objects
satisfy the condition $(\CK)$. The functors $\K$ and $\C$
establish a categorical equivalence between $\Po$ and $\KPop$
with natural isomorphisms  $\alpha$ and $\beta$.
\end{thm}

Let $\MS$ be the category whose objects are bounded semilattices
and whose morphisms are the algebra homomorphisms.

\begin{defn}
We write $\KMS$ to denote the category whose objects are the
structures $(T, \leq, {\sim},\ce, 0,1)$ which satisfy the
following conditions:
\begin{enumerate}
\item[$\KMo$] $(T, \leq, {\sim},\ce) \in \KPo$. \item[$\KMtw$] $0$
is the first element of $(T,\leq)$ and $1$ is the greatest element
of $(T,\leq)$. \item[$\KMth$] For every $x,y \in T$, if $x\geq
\ce$, then  $x\we y$ exists. \item[$\KMfo$] For every $x,y \in T$,
if $x\geq \ce$, then $(x\we y)\vee \ce = x\we (y\vee \ce)$.
\end{enumerate}
The morphisms of $\KMS$ are maps $g$ between objects of $\KMS$
which preserve the order, the involution and such that for every
$x,y\geq \ce$, $g(x\we y) = g(x) \we g(y)$.
\end{defn}

We write $\KMSp$ to denote the full subcategory of $\KMS$ whose
objects satisfy the condition $(\CK)$. Note that in presence of
the condition $\KMo$ we can replace the condition $\KMth$ by the
following condition: $(x\vee \ce)\we y$ exists for every $x$, $y$.
Also note that in presence of the conditions $\KMo$ and $\KMth$,
the condition $\KMfo$ can be replaced by the condition $((x\vee
\ce)\we y)\vee \ce = (x\vee \ce)\we (y\vee \ce)$ for every $x$,
$y$.

Recall that $\MS$ is the category whose objects are bounded
semilattices and whose morphisms are the algebra homomorphisms
between them.

\begin{cor} \label{cms}
The functors $\K$ and $\C$ establish a categorical equivalence between
$\MS$ and $\KMSp$ with natural isomorphisms $\alpha$ and  $\beta$.
\end{cor}

\begin{proof}
Let $H\in \MS$. The condition $\KMo$ for $\K(H)$ follows from
Theorem \ref{Kfp}. We also have that $(0,1)$ is the first element
of $\K(H)$ and $(1,0)$ is the last element of $\K(H)$, i.e., we have the
condition $\KMtw$. Let $x,y \in \K(H)$ with $x\geq \ce$. Then
there are $a,b,d\in H$ such that $x = (a,0)$, $y = (b,d)$ and
$b\we d = 0$. Since in particular $a\we b$ exists, then it follows
from Lemma \ref{l1} that  $x \we y$ exists and $x\we y = (a \we
b,d)$. Then we have proved $\KMth$. Again taking into account
Lemma \ref{l1} we deduce that $((a,0)\we (b,d))\vee (0,0) = (a\we
b,0)$. The mentioned lemma also implies that $(b,d) \vee (0,0) =
(b,0)$ and $(a,0) \we (b,0) = (a\we b,0)$. Thus, $(x\we y)\vee \ce
= x\we (y\vee \ce)$, which is $\KMfo$. Therefore $\K(H) \in \KMS$. It
is immediate that if $T\in \KMS$, then $\C(T) \in \MS$. The rest of
the proof follows from Theorem \ref{Kfp}.
\end{proof}

\section{The variety of hemi-implicative semilattices (lattices)}
\label{s4}

In this section we recall definitions and properties about the
algebras we will consider later: hemi-implicative semilattices
(lattices) \cite{SM2}, Hilbert algebras with infimum \cite{Fig},
implicative semilattices \cite{N} and semi-Heyting algebras
\cite{Sanka}.

\begin{defn}
A \emph{hemi-implicative semilattice} is an algebra $(H,\we,\ra,
1)$ of type $(2,2,0)$  which satisfies the following conditions:
\begin{enumerate}
\item[$\Won$] $(H,\we,1)$ is an upper bounded semilattice,
\item[$\Wtw$] for every $a,b,d \in H$, if $a\leq b\ra d$ then $a
\we b \leq  d$, \item[$\Wth$]$a\ra a = 1$ for every $a\in H$.
\end{enumerate}
A \emph{bounded hemi-implicative semilattice} is an algebra
$(H,\we,\ra,0,1)$ of type $(2,2,0,0)$ such that $(H,\we,\ra,1)$ is
a hemi-implicative semilattice and $0$ is the first element with
respect to the order. A \emph{hemi-implicative lattice} is an
algebra $(H,\we,\vee,\ra,0,1)$ of type $(2,2,2,0,0)$ such that
$(H,\we,\vee,0,1)\in \DL$ and $(H,\we,\ra,1)$ is a
hemi-implicative semilattice.
\end{defn}

Hemi-implicative semilattices were called weak implicative
semilattices in \cite{SM2}. We write $\hIS$ for the category of
bounded hemi-implicative semilattices and $\hBDL$ for the category
of hemi-implicative lattices.

\begin{rem}
If $(H,\we)$ is a semilattice and $\ra$ a binary operation, then
$H$ satisfies $\Wtw$ if and only if for every $a,b \in A$, $a\we
(a\ra b)\leq b$. Therefore, the class of hemi-implicative
semilattices (lattices) is a variety \cite{SM2}.
\end{rem}

The variety of Hilbert algebras is the algebraic counterpart of
the implicative fragment of Intuitionistic Propositional Logic.
These algebras were introduced in the early 50's by Henkin and
Skolem for some investigations on the implication in
intuitionistic logic and other non-classical logics \cite{R}. In
the 1960s, they were studied especially by Horn and Diego
\cite{D}.

\begin{defn}
A \emph{Hilbert algebra} is an algebra $(H,\ra, 1)$ of type $(2,0)$
 that satisfies the following conditions:
\begin{enumerate}[\normalfont 1)]
\item $a\ra (b\ra a) = 1$. \item $a\ra (b\ra d) = (a\ra b)\ra
(a\ra d)$. \item If $a\ra b = b\ra a = 1$, then $a = b$.
\end{enumerate}
\end{defn}

It is well a known fact that Hilbert algebras form a variety. In
every Hilbert algebra we have the partial order defined by $a\leq
b$ if and only if $a\ra b = 1$. In particular, $a\ra a = 1$ for
every $a$.

\begin{ex} \label{ej1}
In any poset $(H,\leq)$ with last element $1$ it is possible to
define the following binary operation:
\[
a \ra b =
\begin{cases}
 1, &\text{if $a \leq b$;}\\
 b,  &\text{if $a\nleq b$.}
 \end{cases}
\]
The structure $(H,\ra ,1)$ is a Hilbert algebra.
\end{ex}

For the following definition see \cite{Fig}.

\begin{defn}
An algebra $(H,\we,\ra,1)$ is a \emph{Hilbert algebra with infimum} if
the following conditions hold:
\begin{enumerate}[\normalfont 1)]
\item $(H,\ra,1)$ is a Hilbert algebra, \item $(H,\we,1)$ is an
upper bounded semilattice, \item For every $a,b \in H$, $a\leq b$ if and
only if $a\ra b = 1$, where $\leq$ is the semilattice order.
\end{enumerate}
An algebra $(H,\we,\ra,0,1)$ of type $(2,2,0,0)$ is a
\textit{bounded Hilbert algebra with infimum} if $(H,\we,\ra,1)$
is a Hilbert algebra with infimum and $0$ is the first element
with respect to the induced order.
\end{defn}

In \cite{Fig} it is proved that the class of Hilbert algebras with
infimum is a variety. We note that this result also follows from the
results given by P. M. Idziak in \cite{I} for BCK-algebras with
lattice operations. The following proposition can be found in
\cite[Theorem 2.1]{Fig}.

\begin{prop} \label{pHil}
Let $(H,\we,\ra,1)$ be an algebra of type $(2,2,0)$. Then
$(H,\we,\ra,1)$ is a Hilbert algebra with infimum if and only if
for every $a,b,d\in H$ the following conditions hold:
\begin{enumerate} [\normalfont a)]
\item $(H,\ra,1)$ is a Hilbert algebra, \item $(H,\we,1)$ is an upper
bounded semilattice, \item $a\we (a\ra b) = a\we b$, \item $a \ra
(b\we d) \leq (a\ra b) \we (a\ra d)$.
\end{enumerate}
\end{prop}

In every Hilbert algebra with infimum we have  $a\ra a = 1$ and
$a\we (a\ra b) \leq b$, so the variety of Hilbert algebras with
infimum is a subvariety of the variety of hemi-implicative
semilattices.

We will write $\Hil$ for the category whose objects are bounded
Hilbert algebras with infimum and whose morphisms are the
corresponding algebra homomorphisms. Clearly $\Hil$ is a full
subcategory of $\hIS$.

\begin{defn}
An \emph{implicative semilattice} is an algebra $(H, \we, \ra)$ of
type $(2,2)$ such that $(H,\we)$ is a semilattice, and for every
$a,b,d\in H$ we have that $a\we b \leq d$ if and only if $a\leq
b\ra d$.
\end{defn}

Implicative semilattices have a greatest element, denoted by $1$.
In this paper we shall include the constant $1$ in the language of
the algebras. Implicative semilattices are the algebraic models of
the implication-conjunction fragment of Intuitionistic
Propositional Logic. For more details about these algebras see
\cite{Cu}.

An algebra $(H,\we,\ra,0,1)$ of type $(2,2,0,0)$ is a
\emph{bounded implicative semilattice} if $(H,\we,\ra,1)$ is an
implicative semilattice and $0$ is the first element with respect
to the order. We write $\IS$ for the category whose objects are
bounded implicative semilattices and whose morphisms are the
corresponding algebra homomorphisms. We have that $\IS$ is a full
subcategory of $\hIS$.

It is part of the folklore of the subject that the class of
implicative semilattices is a variety. There are many ways to
axiomatize the variety of implicative semilattices. In the
following lemma we propose a possible axiomatization that will
play an important role in the next section.

\begin{lem} \label{lis}
Let $(H,\we,1)$ be an upper bounded semilattice and $\ra$ a binary
operation on $H$. The following conditions are equivalent:
\begin{enumerate}[\normalfont (a)]
\item For every $a,b,d \in H$, $a\leq b\ra d$ if and only if $a\we
b \leq d$. \item For every $a,b,d \in H$ the following conditions
hold:
\begin{enumerate}[\normalfont 1)]
\item $a\we (a\ra b) \leq b$, \item $a\ra a = 1$, \item $a\ra
(b\we d) = (a\ra b) \we (a\ra d)$, \item $a\leq b \ra (a\we b)$.
\end{enumerate}
\end{enumerate}
\end{lem}

\begin{proof}
Assume the conditions 1), 2), 3) and 4) of (b). It follows from 1)
that if $a\leq b\ra d$, then $a\we b \leq d$. Suppose now that
$a\we b \leq d$. It follows from 3) that $b\ra (a\we b) \leq b\ra
d$. But by 4) we have that $a\leq b \ra (a\we b)$, so $a\leq b\ra
d$. Then,  $a\leq b\ra d$ if and only if $a\we b \leq d$. For the
converse of this property see \cite{N}.
\end{proof}

\begin{rem} \label{isaha}
A  moment of reflection shows that implicative semilattices are
Hilbert algebras with infimum where the implication is the right
residuum of the infimum, or equivalently, where the following
equation holds \cite{Fig}: $a\leq b \ra (a\we b)$. Alternatively,
it follows from Lemma \ref{lis} that an implicative semilattice is
a hemi-implicative semilattice which satisfies $a\ra (b\we d) =
(a\ra b) \we (a\ra d)$ and $a\leq b \ra (a\we b)$ for every
$a,b,d$.
\end{rem}

Semi-Heyting algebras were introduced by H.P. Sankappanavar in
\cite{Sanka} as an abstraction of Heyting algebras. These algebras
share with Heyting algebras the following properties: they are
pseudocomplemented and distributive lattices and their  congruences
are determined by the lattice filters.

\begin{defn}
An algebra $(H, \we, \vee, \ra, 0, 1)$ of type $(2,2,2,0,0)$ is a
\emph{semi-Heyting algebra} if the following conditions hold for
every $a,b,d$ in $H$:
\begin{enumerate}
\item[$\SHon$] $(H, \we, \vee, 0, 1)\in \DL$, \item[$\SHtw$] $a\we
(a\ra b) = a \we b$, \item[$\SHth$] $a\we (b\ra d) = a \we ((a\we
b) \ra (a\we d))$, \item[$\SHfo$] $a\ra a = 1$.
\end{enumerate}
\end{defn}

We write $\SH$ for the category of semi-Heyting algebras. A
semi-Heyting algebra can be seen as a hemi-implicative lattice
which satisfies $\SHtw$ and $\SHth$. Therefore, $\SH$ is a full
subcategory of $\hBDL$.

\begin{rem}
Implicative semilattices satisfy the inequality $a\leq b \ra (a\we
b)$ because $a\we b \leq a\we b$, or simply by Lemma \ref{lis}.
Semi-Heyting algebras also satisfy the inequality $a\leq b \ra
(a\we b)$. This fact follows from $\SHth$ and $\SHfo$ in the
following way:
\[
\begin{array}
[c]{lllll}
a \we (b \ra (a\we b)) & = &  a\we ((a\we b) \ra (a\we b)) &  & \\
& = & a\we 1&  & \\
& = & a,& &
\end{array}
\]
which means that $a\leq b \ra (a\we b)$.
\end{rem}

In the following example we will show the following facts: $\Hil$
is a proper subvariety of $\hIS$ and $\SH$ is a proper subvariety
of $\hBDL$.

\begin{ex}
Let $H$ be the chain of three elements with $0<a<1$. We define on
$H$ the following binary operation:
\[%
\begin{tabular}
[c]{c|ccc}%
$\ra$ & $0$ & $a$ & $1$\\\hline
$0$ & $1$ & $a$ & $1$\\
$a$ & $0$ & $1$ & $1$\\
$1$ & $0$ & $0$ & $1$%
\end{tabular}
\]
Straightforward computations show that $(H, \we, \vee,\ra, 0,1)
\in \hBDL$. In particular, $(H,\we,\ra,0,1) \in \hIS$. However,
$(H,\we,\ra,0,1) \notin \Hil$ and $(H,\we,\vee,\ra,0,1) \notin
\SH$ because $1\we (1\ra a) \neq 1\we a$.
\end{ex}

It is a known fact that the variety $\IS$ is properly included in
$\Hil$. We give an example that shows it. Let $H$ be the universe
of the boolean lattice of four elements, where $a$ and $b$ are the
atoms. Then $(H,\we,\ra,0,1) \in \Hil$, where $\ra$ is the
operation defined in Example \ref{ej1}. Since $a\ra 0 = 0$ and
$0\neq b$, then $(H,\we,\ra,0,1) \notin \IS$.

It is also a known fact that $\He$ is a proper subvariety of
$\SH$. We also provide an example that shows it. Consider the
chain of two elements. We define the following binary operation:
\[%
\begin{tabular}
[c]{c|cc}%
$\ra$ & $0$ & $1$\\\hline
$0$   & $1$ & $0$\\
$1$   & $0$ & $1$%
\end{tabular}
\]
Then $(H,\we,\vee,\ra,0,1)\in \SH$. Since $0\ra 1 = 0$ and $0 \neq
1$, then $(H,\we,\vee,\ra,0,1)$ is not a Heyting algebra.

The following diagrams show the relations among the categories
defined in this section:
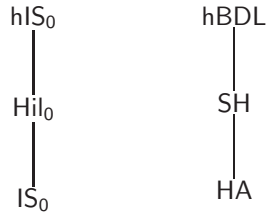
\begin{equation*}
\\\\\\\\\\\\\\\\\\\\\\\\\\\\\\\\\\\\\\\\\\\\\\\\\\\\\\
\xymatrix@1{
& \hIS &\\
& \ar@{-}[u] \Hil &\\
& \ar@{-}[u] \IS &\\
 }
\\\\\\\\\\\\\\\\\\\\\\\\\\\\\\\\\\\\\\\\\\\\\\\\\\\\\\
\xymatrix@1{
& \hBDL &\\
& \ar@{-}[u] \SH &\\
& \ar@{-}[u] \He &\\
 }
\end{equation*}

In \cite{CCSM}  an extension of Kalman's functor was studied for
the variety of algebras with implication $(H,\we,\vee, \ra, 0,1)$
which satisfy $a\we (a\ra b)\leq b$ for every $a$, $b$.

\section{Kalman's construction for $\hIS$ and $\hBDL$} \label{s5}

The fact that Kalman's construction can be extended consistently
to Heyting algebras led us to believe that some of the picture
could be lifted to the varieties $\hIS$ and $\hBDL$. More
precisely, it arises the natural question of wether  is it
possible to find some category $\KhIS$ in order to obtain an
equivalence between $\hIS$ and some full subcategory of $\KhIS$,
making the following diagram commute:
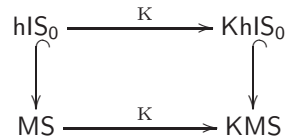
\[
 \xymatrix{
   \hIS \ar[rr]^\K    \ar@{^{(}->}[d]  & & \KhIS  \ar@{^{(}->}[d] \\
   \MS              \ar[rr]^\K    &   & \KMS  \\
   }
\]
Similarly, it arises the question of wether is it possible to find
some category $\KhBDL$ in order to obtain an equivalence between
$\hBDL$ and certain full subcategory of $\KhBDL$, making the
following diagram commute:
\[
 \xymatrix{
   \hBDL \ar[rr]^\K    \ar@{^{(}->}[d]  & & \KhBDL  \ar@{^{(}->}[d] \\
   \DL              \ar[rr]^\K    &   & \cKl  \\
   }
\]
In this section,  we answer these questions in the positive.
Moreover, we extend Kalman's functor to the categories $\Hil$,
$\IS$ and $\SH$.

The aim of Section \ref{s3} was to obtain a categorical
equivalence between $\MS$ and $\KMSp$ (Corollary \ref{cms})
 to be applied in the present section to the category $\hIS$.
For the case of the category $\hBDL$ we will also use Theorem
\ref{Kce}, which establishes an equivalence between $\DL$ and
$\cKlp$.

\subsection{Kalman's construction for $\hIS$}

Let $H\in \hIS$. We write $\ra$ for the implication of $H$ and define
a binary operation on $\K(H)$ (also denoted   $\ra$) by
\begin{equation} \label{eqi1}
(a,b) \ra (d,e): = ((a\ra d) \we (e\ra b), a\we e).
\end{equation}
This definition is motivated by Remark \ref{imp}. Note that since
$a\we b = d\we e = 0$, then $(a\ra d)\we (e\ra b) \we a\we e = 0$
because $a\we (a\ra d) \leq d$ and $d\we e = 0$. Hence, $(a,b) \ra
(d,e) \in \K(H)$. The next definition is motivated by the original
Kalman's construction.

\begin{defn}
We denote by $\KhIS$ the category whose objects are the structures
$(T, \leq, {\sim}, \ra, \ce, 0,1)$ such that $(T, \leq,
{\sim},\ce, 0,1)\in \KMS$ and $\ra$ is a binary operation on $T$
which satisfies the following conditions for every $x,y \in T$:
\begin{enumerate}
\item[$\Ko$] $\ce \leq x \ra (y\vee \ce)$, \item[$\Ktw$] $x \we
((x\vee \ce)\ra (y\vee \ce))\leq y \vee \ce$, \item[$\Kth$] $x \ra
x = 1$,
\item[$\Kfo$] $(x\ra y)\we \ce = ({\sim} x \we \ce) \vee
(y\we \ce)$,
\item[$\Kfi$] $(x\ra {\sim} y) \vee \ce = ((x\vee \ce)
\ra({\sim} y \vee \ce)) \we ((y\vee \ce) \ra ({\sim} x \vee \ce))$.
\end{enumerate}
The morphisms of $\KhIS$ are the morphisms $g$ of $\KMS$ which
satisfy the condition $g(x\ra y) = g(x)\ra g(y)$ for every $x$,
$y$.
\end{defn}

In what follows we will prove that if $H \in \hIS$, then $\K(H) \in
\KhIS$, where the binary operation $\ra$ in $\K(H)$ is that defined
in (\ref{eqi1}).

\begin{prop} \label{Kl}
Let $H \in \hIS$. Then $\K(H) \in \KhIS$. Furthermore, $\K$
extends to a functor from $\hIS$ to $\KhIS$, which we also denote
by $\K$.
\end{prop}

\begin{proof}
Throughout this proof we use lemmas \ref{l1} and \ref{la}. Recall
that $c = (0,0)$.

Let $(a,b), (d,e) \in \K(H)$. In particular, $(d,e) \vee \ce =
(d,0)$. Then
\[
\begin{array}
[c]{lllll} (a,b)\ra ((d,e) \vee \ce)& = & (a,b)\ra (d,0)&  & \\
 & = & ((a\ra d)\we (0\ra b),0)&  & \\
 & \succeq & \ce.&  &
\end{array}
\]
Thus we have proved the condition $\Ko$. In order to prove $\Ktw$
we make the following computation:
\[
\begin{array}
[c]{lllll}
(a,b) \we (((a,b) \vee \ce) \ra ((d,e) \vee \ce))& = &(a,b) \we ((a,0)\ra (d,0))&  & \\
 & = & (a,b) \we (a\ra d, 0)&  & \\
 & = & (a\we (a\ra d),b) &  &\\
 & \preceq & (d,b) &  & \\
 & \preceq & (d,0) &  & \\
 & = & (d,e) \vee \ce. &  & \\
\end{array}
\]
The proof of the condition $\Kth$ is immediate. In order to prove
$\Kfo$, note that $((a,b)\ra (d,e))\we \ce = (0,a\we e)$ and
\[
\begin{array}
[c]{lllll}
({\sim} (a,b)\we \ce) \vee ((d,e) \we \ce)& = &(0,a) \vee (0,e)&  & \\
 & = & (0,a\we e).&  &
\end{array}
\]
Hence, we have that
\[
((a,b)\ra (d,e))\we \ce = ({\sim} (a,b)\we \ce) \vee ((d,e) \we
\ce).
\]
Finally we shall prove $\Kfi$. First note that
\[
\begin{array}
[c]{lllll}
((a,b) \ra {\sim}(d,e)) \vee \ce& = &((a,b) \ra (e,d))\vee \ce&  & \\
 & = & ((a\ra e)\we (d\ra b),0).&  &
\end{array}
\]
Then,
\begin{equation} \label{K5on}
((a,b) \ra {\sim}(d,e)) \vee \ce = ((a\ra e)\we (d\ra b),0).
\end{equation}
On the other hand,
\[
\begin{array}
[c]{lllll}
((a,0)\ra (e,0))\we ((d,0)\ra (b,0)) & = &(a\ra e,0) \we (d\ra b,0)&  & \\
& = & ((a\ra e)\we (d\ra b),0).&  &
\end{array}
\]
Hence,
\begin{equation} \label{K5tw}
((a,0)\ra (e,0))\we ((d,0)\ra (b,0)) = ((a\ra e)\we (d\ra b),0).
\end{equation}
Since $(a,b) \vee \ce = (a,0)$, ${\sim}(d,e) \vee \ce = (e,0)$,
$(d,e) \vee \ce = (d,0)$,  and ${\sim}(a,b) \vee \ce = (b,0)$,  then
it follows from (\ref{K5on}) and (\ref{K5tw}) that the condition
$\Kfi$ holds. Thus, $\K(H)\in \KhIS$.

Let $f:H\ra G$ be a morphism in $\hIS$. Straightforward
computations show that $\K(f)$ preserves the implication
operation, which implies that $\K(f)$ is a morphism in $\KhIS$.
\end{proof}

\begin{prop} \label{p2}
If $(T, \leq, {\sim}, \ra, \ce, 0,1)\in \KhIS$, then
$(\C(T),\we,\ra,\ce,1) \in \hIS$. Furthermore, $\C$ extends to a
functor from $\KhIS$ to $\hIS$, which we also denote by $\C$.
\end{prop}

\begin{proof}
We have that $\C(T)$ is closed under the operation $\ra$. In order
to prove it, let $x,y \geq \ce$. By $\Ko$ we have that
\[
\begin{array}
[c]{lllll}
\ce & \leq &(x\vee \ce) \ra (y\vee \ce)&  & \\
& = & x \ra y,&  &
\end{array}
\]
so $x\ra y \in \C(T)$. Thus, the restriction of $\ra$ to $\C(T)$ is indeed an operation on $\C(T)$. Let
$x,y \geq \ce$. It follows from $\Ktw$ that $x\we (x\ra y)\leq y$
and it follows from $\Kth$ that $x\ra x = 1$. Then
$(\C(T),\we,\ra,\ce,1) \in \hIS$. The rest of the proof is
immediate.
\end{proof}

\begin{rem} \label{a}
For $H\in \hIS$ we have that $\alpha_{H}(a\ra b) = \alpha_{H}(a)
\ra \alpha_{H}(b)$ for every $a$, $b\in H$. Moreover, $\alpha_H$
is an isomorphism in $\hIS$.
\end{rem}

For the case of $T\in \KhIS$ we will prove that $\beta_T$
preserves the implication.

\begin{lem} \label{b}
Let $T\in \KhIS$. Then $\beta_T$ is injective and a morphism in
$\KhIS$. Moreover, if $T$ satisfies $(\CK)$ then $\beta_T$ is an
isomorphism in $\KhIS$.
\end{lem}

\begin{proof}
We need to prove that $\beta_T(x\ra y) = \beta_T(x) \ra
\beta_T(y)$ for every $x,y$. In an equivalent way, we need to
prove that $\beta_T(x\ra {\sim} y) = \beta_T(x) \ra \beta_T({\sim}
y)$ for every $x,y$. It follows from $\Kfo$ and $\Kfi$ that
\[
\begin{array}
[c]{lllll}
\beta_T(x\ra {\sim} y) & = & ((x\ra {\sim} y) \vee \ce,{\sim} (x\ra {\sim} y) \vee \ce)) &  & \\
& = & (((x\vee \ce) \ra({\sim} y \vee \ce)) \we ((y\vee \ce) \ra
({\sim} x \vee \ce)), (x\vee \ce) \we (y \vee \ce)) &  & \\
 & = & \beta_T(x) \ra \beta_T({\sim} y).&  &
\end{array}
\]
\end{proof}

We write $\KhISp$ for the full subcategory of $\KhIS$ whose
objects satisfy $(\CK)$. The proof of the following theorem
follows from Corollary \ref{cms}, Proposition \ref{Kl},
Proposition \ref{p2}, Remark \ref{a} and Lemma \ref{b}.

\begin{thm} \label{pt}
The functors $\K$ and $\C$ establish a categorical equivalence
between $\hIS$ and $\KhISp$ with natural isomorphisms $\alpha$ and
$\beta$.
\end{thm}

Let $T \in \KhIS$. We define the following condition for every
$x,y \in T$:
\begin{enumerate}
\item[$\Ksi$] $x \leq (y\vee \ce) \ra ((x\vee \ce)\we (y\vee
\ce))$.
\end{enumerate}

The next lemma is the motivation to consider the condition $\Ksi$.

\begin{lem} \label{CK-K6}
If $H\in \hIS$ satisfies the inequality $a\leq b \ra (a\we b)$ for
every $a, b$,  then $\K(H)$ satisfies $\Ksi$.
\end{lem}

\begin{proof}
First note that for every $(a,b) \in \K(H)$, $(a,b) \vee \ce =
(a,0)$. In order to prove $\Ksi$ we make the following
computation:
\[
\begin{array}
[c]{lllll}
(d,0) \ra ((a,0)\we (d,0)) & = &  (d,0) \ra (a\we d,0) &  & \\
& = & (d \ra (a\we d),0)&  & \\
& \succeq & (a,b).& &
\end{array}
\]
Hence, we obtain  $\Ksi$.
\end{proof}

The following lemma will play an important role in this paper.

\begin{lem} \label{lCK}
If $T \in \KhIS$ satisfies $\Ksi$, then $T$ satisfies $(\CK)$.
\end{lem}

\begin{proof}
Let $x,y \geq \ce$ such that $x\we y = \ce$. Taking into account
$\KMth$ we can define $z = (y \ra {\sim} y) \we x$. It follows
from $\Kfo$ that
\[
\begin{array}
[c]{lllll}
z\we \ce & = &  ((y\ra {\sim} y) \we \ce)\we x &  & \\
& = & (({\sim} y \we \ce) \vee ({\sim} y \we \ce)) \we x&  & \\
& = & ({\sim} y \we \ce) \we x&  & \\
& = & {\sim} y \we x&  & \\
& = & {\sim} y.&  &
\end{array}
\]
Hence, ${\sim} z \vee \ce = y$. In order to prove that $z\vee \ce
= x$, we use the conditions $\KMth$, $\KMfo$, $\Kfi$, $\Ksi$ and
the fact that $x\we y = \ce$ as follows:
\[
\begin{array}
[c]{lllll}
z\vee \ce & = &  (x\we (y\ra {\sim} y))\vee \ce &  & \\
& = & x\we ((y\ra {\sim} y) \vee \ce) &  &\\
& = & ((y\vee \ce) \ra ({\sim} y \vee \ce))\we x&  & \\
& = & (x\vee \ce) \we ((y\vee \ce) \ra \ce)&  & \\
& = & (x \vee \ce) \we ((y\vee \ce) \ra (x\we y)) &  & \\
& = & (x \vee \ce) \we ((y\vee \ce) \ra ((x\vee \ce)\we (y\vee \ce))) & &\\
& = & x \vee \ce &  &\\
& = & x. &  &
\end{array}
\]
Therefore, $z\vee \ce = x$.
\end{proof}

\subsection{Kalman's construction for $\Hil$}

We write $\KHil$ for the full subcategory of $\KhIS$ whose objects
satisfy the following conditions for every $x,y,z$:
\begin{enumerate}
\item[$\KHon$] $(x\vee \ce)\ra (y\ra (x\vee \ce)) = 1$,
\item[$\KHtw$] $x \ra ((y\vee \ce) \ra (z\vee \ce)) = (x \ra
(y\vee \ce)) \ra (x\ra (z \vee \ce))$, \item[$\KHth$] If $x\ra y =
y \ra x = 1$, then $x = y$, \item[$\KHfo$] $x\we ((x\vee \ce)\ra
(y\vee \ce)) = x\we (y\vee \ce)$, \item[$\KHfi$] $x\ra ((y\vee
\ce) \we (z\vee \ce)) \leq (x\ra (y\vee \ce))\we (x\ra (z\vee
\ce))$.
\end{enumerate}

\begin{ex} \label{KHil}
In every centered Kleene algebra $(T,\we,\vee,{\sim},\ce,0,1)$ it
is possible to define a binary operation, that we denote by $\ra$,
as follows:
\[
x \ra y =
 \begin{cases}
 1, &\text{if $x \vee \ce \leq y \vee \ce$ and $x\we \ce \leq y\we \ce$;}\\
 {\sim} x \vee (y\we \ce),  &\text{if $x \vee \ce \leq y \vee \ce$ and $x\we \ce \nleq y\we \ce $;}\\
 y \vee ({\sim} x \we \ce), &\text{if $x \vee \ce \nleq y \vee \ce$ and $x\we \ce \leq y\we \ce$;}\\
 ((y\vee \ce)\we {\sim} x) \vee (({\sim} x \vee \ce)\we y),  &\text{if  $x\vee \ce \nleq y\vee \ce$ and $x\we \ce \nleq y\we \ce$.}
 \end{cases}
\]
It is possible to prove that $(T,\leq, {\sim},\ce,0,1)\in \KMS$,
and it is not difficult to see that $(T,{\sim},\ra,\ce,0,1) \in
\KHil$.

By endowing  the centered Kleene algebra given in \cite[Example
2.5]{CCSM} with the binary operation $\ra$ just defined we obtain
an example of an object of $\KHil$ which does not satisfy the
condition $(\CK)$.
\end{ex}

\begin{lem} \label{lHa}
\begin{enumerate} [\normalfont (a)]
\item[]
\item If $H\in \Hil$, then $\K(H) \in \KHil$. \item If $T\in \KHil$,
then $\C(T) \in \Hil$.
\end{enumerate}
\end{lem}

\begin{proof}
Let $H\in \Hil$ and take $(a,b)$, $(d,e)$ and $(f,g)$ in $\K(H)$.
In what follows we will use Proposition \ref{pHil}.

Taking into account that $a\ra (d\ra a) = 1$ we obtain
\[
\begin{array}
[c]{lllll}
(a,0) \ra ((d,e) \ra (a,0)) & = &  (a,0) \ra (d\ra a,0) &  & \\
& = & (a \ra (d\ra a) ,0)&  & \\
& = & (1,0),& &
\end{array}
\]
which is the condition $\KHon$.

Since $a\ra (d\ra f) = (a\ra d) \ra (a\ra f)$, then
\[
\begin{array}
[c]{lllll}
(a,b) \ra ((d,0) \ra (f,0)) & = &  (a,b) \ra (d\ra f,0) &  & \\
& = & (a\ra (d\ra f), 0)&  & \\
& = & ((a\ra d) \ra (a \ra f), 0)&  & \\
& = & (a\ra d,0) \ra (a \ra f), 0)&  & \\
& = & ((a,b) \ra (d,0)) \ra ((a,b)\ra (f,0)).& &
\end{array}
\]
Hence, we have proved $\KHtw$.

In order to prove $\KHth$ suppose that $(a,b) \ra (d,e) = (d,e)
\ra (a,b) = (1,0)$, so $a\ra d = d\ra a = 1$ and $b\ra e = e \ra b
= 1$. Then $a = d$ and $b = e$, i.e., $(a,b) = (d,e)$, which was
our aim.

The condition $\KHfo$ is a consequence of the equality $a\we (a\ra
d) = a\we d$. Indeed,
\[
\begin{array}
[c]{lllll}
(a,b) \we ((a,0) \ra (d,0)) & = &  (a,b) \we (a\ra d,0) &  & \\
& = & (a\we (a\ra d), b)&  & \\
& = & (a\we d,b)&  & \\
& = & (a,b) \we (d,0).& &
\end{array}
\]

Finally, we will prove $\KHfi$. By the condition $a\ra (d\we f)
\leq (a\ra d)\we (a\ra f)$ we have that
\[
\begin{array}
[c]{lllll}
(a,b) \ra ((d,0)\we (f,0)) & = &  (a,b) \ra (d\we f,0) &  & \\
& = & (a \ra (d\we f), 0)&  & \\
& \preceq & ((a\ra d)\we (a\ra f), 0)&  & \\
& = & (a\ra d, 0) \we (a\ra f,0)&  & \\
& = & ((a,b) \ra (d,0)) \we ((a,b) \ra (f,0)).& &
\end{array}
\]
Then $\K(H) \in \KHil$. Finally, it follows from Proposition
\ref{pHil} that if $T\in \KHil$, then $\C(T) \in \Hil$
\end{proof}

We write $\KHilp$ for the full subcategory of $\KHil$ whose
objects satisfy $(\CK)$. The following corollary follows from
Theorem \ref{pt} and Lemma \ref{lHa}.

\begin{cor} \label{corhil}
The functors $\K$ and $\C$ establish a categorical equivalence
between  $\Hil$ and $\KHilp$ with natural isomorphisms $\alpha$
and $\beta$.
\end{cor}

\subsection{Kalman's construction for $\IS$}

We write $\KIS$ for the full subcategory of $\KhIS$ whose objects
$T$ satisfy the condition $\Ksi$ and the following additional
condition for every $x,y \in T$:
\begin{enumerate}
\item[$\Kse$] $x \ra ((y\vee \ce) \we (z\vee \ce)) = (x \ra (y\vee
\ce))\we (x \ra (z\vee \ce))$.
\end{enumerate}

\begin{lem} \label{lis2}
\begin{enumerate} [\normalfont (a)]
\item[]
\item If $H\in \IS$, then $\K(H) \in \KIS$. \item If $T\in \KIS$,
then $\C(T) \in \IS$.
\end{enumerate}
\end{lem}

\begin{proof}
Let $H\in \IS$. The fact that $\K(H)$ satisfies $\Ksi$ follows
from lemmas \ref{lis} and \ref{CK-K6}. On the other hand, it
follows from Lemma \ref{lis} that
\[
\begin{array}
[c]{lllll}
(a,b)\ra ((d,0)\we (f,0)) & = &  (a,b) \ra (d\we f,0) &  & \\
& = & (a \ra (d\we f),0)&  & \\
& = & ((a\ra d)\we (a\ra f),0)&  & \\
& = & ((a,b)\ra (d,0))\we ((a,b)\ra (f,0)).&  &
\end{array}
\]
Thus, we have the condition $\Kse$. Then $\K(H) \in \KIS$.

The fact that if $T\in \KIS$, then $\C(T) \in \IS$ is also
consequence of Lemma \ref{lis}.
\end{proof}

The following corollary follows from Theorem \ref{pt}, Lemma
\ref{lCK} and Lemma \ref{lis2}.

\begin{cor} \label{coris}
The functors $\K$ and $\C$ establish a categorical equivalence
between $\IS$ and $\KIS$ with  natural isomorphisms $\alpha$ and
$\beta$.
\end{cor}

Since $\IS$ is a full subcategory of $\Hil$, it follows from
corollaries \ref{corhil} and \ref{coris} that $\KIS$ is a full
subcategory of $\KHilp$.

\subsection{Kalman's construction for $\hBDL$}

In what follows we define a category which will be related with
the category $\hBDL$.

\begin{defn}
We write $\KhBDL$ for the category whose objects are the algebras
$(T,\we,\vee,$ $\ra, {\sim},\ce,0,1)$ of type $(2,2,2,1,0,0,0)$
such that $(T,\we,\vee, {\sim},\ce,0,1)\in \cKl$ and the
conditions $\Ko$, $\Ktw$, $\Kth$, $\Kfo$ and $\Kfi$ are satisfied.
The morphisms of the category are the corresponding algebra
homomorphisms.
\end{defn}

By the Example \ref{KHil}, in every centered Kleene algebra
$(T,\we,\vee,0,\ce,0,1)$ we can define a binary operation $\ra$
such that $(T,\we,\vee,\ra,\ce,0,1) \in \KhBDL$. In particular, if
$(T,\we,\vee,0, \ce,0,1)$ is the centered Kleene algebra given in
\cite[Example 2.5]{CCSM}, then $(T,\we,\vee,\ra,\ce,0,1)\in
\KhBDL$, where $\ra$ is the implication considered in Example
\ref{KHil}. It is immediate that $(T,\we,\vee,\ce,0,1)$ does not
satisfy the condition ($\CK$).

Note that $(H,\we,\vee,\ra, 0,1) \in \hBDL$ if and only if
$(H,\we,\vee,0,1) \in \DL$ and $(H,\we,\ra,0,1) \in \hIS$. Also
note that $(T,\we,\vee,\ra, {\sim},\ce,0,1) \in \KhBDL$ if and
only if $(T,\we,\vee, {\sim},\ce,0,1)\in \cKl$ and $(T, \leq,
{\sim}, \ra, \ce, 0,1) \in \KhIS$. We write $\KhBDLp$ for the full
subcategory of $\KhBDL$ whose objects satisfy $(\CK)$.

\begin{thm} \label{ptb}
The functors $\K$ and $\C$ establish a categorical equivalence between
$\hBDL$ and $\KhBDLp$ with natural isomorphisms
$\alpha$ and $\beta$.
\end{thm}

\begin{proof}
It follows from Theorem \ref{Kce} and Theorem \ref{pt}.
\end{proof}

\subsection{Kalman's construction for $\SH$}

We write $\KSH$ for the full subcategory of $\KhBDL$ whose objects
satisfy the condition $\KHfo$ and the following additional
condition:
\begin{enumerate}
\item[$\KSHth$] $x \we ((y\vee \ce) \ra (z\vee \ce)) = x
\we(((x\vee \ce) \we (y\vee \ce))\ra ((x\vee \ce)\we (z\vee
\ce)))$.
\end{enumerate}
\newpage

\begin{lem} \label{lSH}
\begin{enumerate} [\normalfont (a)]
\item[]
\item If $H\in \SH$, then $\K(H) \in \KSH$. \item If $T\in \KSH$,
then $\C(T) \in \SH$. \item If $T \in \KSH$, then $T$ satisfies
$\Ksi$. In particular, $T$ satisfies $(\CK)$.
\end{enumerate}
\end{lem}

\begin{proof}
Let $H\in \SH$. The condition $\KHfo$ follows from $\SHtw$ (see
proof of Lemma \ref{lHa}). Let $(a,b)$, $(d,e)$ and $(f,g)$ in
$\K(H)$. Taking into account $\SHth$ we have that
\[
\begin{array}
[c]{lllll}
(a,b)\we ((d,0)\ra (f,0)) & = &  (a,b) \we (d\ra f,0) &  & \\
& = & (a \we (d\ra f),b)&  & \\
& = & (a \we ((a\we d)\ra (a\we f)),b)&  & \\
& = & (a,b) \we ((a\we d)\ra (a\we f),0)&  & \\
& = & (a,b) \we ((a\we d,0) \ra (a\we f,0))&  & \\
& = & (a,b) \we (((a,0)\we (d,0)) \ra ((a,0)\we (f,0))),& &
\end{array}
\]
which is the condition $\KSHth$. Then $\K(H) \in \KSH$.

It is immediate that if $T\in \KSH$ then $\C(T) \in \SH$. In order
to prove that $T$ satisfies $\Ksi$ we will use $\Kth$ and $\KSHth$
as follows:
\[
\begin{array}
[c]{lllll}
x\we ((y\vee \ce) \ra ((x\vee \ce)\we (y\vee \ce))) & = &  x\we ((y\vee \ce) \ra (((x\vee \ce) \we (y\vee \ce)) \vee \ce) &  & \\
& = & x \we (((x\vee \ce) \we (y\vee \ce)) \ra ((x\vee \ce) \we (y\vee \ce))) &  & \\
& = & x\we 1&  & \\
& = & x.& &
\end{array}
\]
Then $x\leq (y\vee \ce) \ra ((x\vee \ce) \ra (y\vee \ce))$, i.e.,
the condition $\Ksi$. Therefore, it follows from Lemma \ref{lCK}
that $T$ satisfies ($\CK$).
\end{proof}

\begin{thm}
The functors $\K$ and $\C$ establish a categorical equivalence between
$\SH$ and  $\KSH$ with  natural isomorphisms
$\alpha$ and $\beta$.
\end{thm}

\begin{proof}
It follows from Theorem \ref{ptb} and Lemma \ref{lSH}.
\end{proof}

\section{Well-behaved congruences in $\KhIS$ and congruences in
$\KhBDL$} \label{s6}

In this section we introduce the concept of the well-behaved
congruences over objects of $\KhIS$. They  are equivalence
relations with some additional properties. We will prove that if
$T\in \KhIS$ and $\theta$ is a well-behaved congruence on $T$,
then it is possible to define on the quotient $T/\theta$ a partial
order and operations so that $T/\theta \in \KhIS$. For $T\in
\KhIS$ we study the relation between the well-behaved congruences
of $T$ and the congruences of $\C(T)$, and in particular for the
cases where $T\in \KHil$ or $T \in \KIS$. For $T\in \KhBDL$ or
$T\in \KSH$ we also study the relation between the congruences of
$T$ and the congruences of $\C(T)$. Finally,  we study the principal
well-behaved congruences of the objects in $\KhIS$, $\KHil$, and
$\KIS$  and the principal congruences of the objects in $\KhBDL$
and $\KSH$. \vspace{1pt}

We start by fixing notation and giving some useful definitions.
Let $X$ be a set, $x \in X$ and $\theta$ an equivalence relation on
$X$. We write $x/\theta$ to indicate the equivalence class of $x$
associated with the equivalence relation $\theta$, and $X/\theta$
to indicate the quotient set of $X$ associated with $\theta$ (i.e.,
the set of equivalence classes). If $T$ is an algebra, we write
$\Con(T)$ to denote the set of as well as the lattice of
congruences of $T$.

\begin{defn} \label{defwbc}
Let $T\in \KhIS$. We say that an equivalence relation $\theta$ of
$T$ is a \emph{well-behaved congruence of $T$} if it satisfies the
following conditions:
\begin{enumerate}
\item[$\Ton$] $\theta \in \Con((T,\ra,{\sim}))$. \item[$\Ttw$] For
$x,y \in T$, $(x,y) \in \theta$ if and only if $(x\vee \ce,y\vee
\ce) \in \theta$ and $({\sim} x \vee \ce, {\sim} y \vee \ce) \in
\theta$. \item[$\Tth$] For $x$, $y$, $z$ and $w$ in $\C(T)$, if
$(x,y)\in \theta$ and $(z,w) \in \theta$, then $(x\we z,y\we w) \in
\theta$.
\end{enumerate}
\end{defn}

Note that the intersection of any family of well-behaved
congruences of $T\in \KhIS$ is a well-behaved congruence;
therefore the  set of well-behaved congruences of $T$ ordered by
the inclusion relation is a complete lattice.

\begin{rem}
The definition of well-behaved congruence can be also given for
algebras of $\KhBDL$. In this case,  if  $T\in \KhBDL$, then every congruence of $T$ is a  well-behaved  congruence.
\end{rem}

In what follows we define a binary relation in $T/\theta$, where
$T\in \KhIS$ and $\theta$ is a well-behaved congruence of $T$.

\begin{defn} \label{wbo}
Let $T \in \KhIS$. If $\theta$ is a well-behaved congruence of $T$,
then we define in $T/\theta$ the following binary relation $\ll_{\theta}$ by:
\[
x/\theta \ll_{\theta} y/\theta\; \textrm{if and only if}\; ((x\vee
\ce) \we (y\vee \ce), x\vee \ce) \in \theta\; \textrm{and}\;
(({\sim} y \vee \ce)\we ({\sim} x \vee \ce), {\sim} y \vee \ce)
\in \theta.
\]
\end{defn}

If there is no ambiguity, we write $\ll$ in place of
$\ll_{\theta}$. Note that the definition given is good, in the
sense that it is independent of the elements selected as representativess of  the
equivalence classes. In order to show it, suppose that $x/\theta
\ll y/\theta$. Let $z\in x/\theta$ and $w\in y/\theta$. Then by
$\Ttw$ we have that $(x\vee \ce, z\vee \ce) \in \theta$, $({\sim}
x\vee \ce, {\sim} z\vee \ce) \in \theta$, $(y\vee \ce, w\vee \ce)
\in \theta$, and $({\sim} z\vee \ce, {\sim} w\vee \ce) \in
\theta$. Hence it follows from $\Tth$ that
\[
((z\vee \ce)\we (w\vee \ce), (x\vee \ce)\we (y\vee \ce)) \in
\theta.
\]
Since, by the assumption,  $((x\vee \ce)\we (y\vee \ce), x\vee
\ce) \in \theta$, and $(x\vee \ce,z\vee \ce) \in \theta$, then
\[
((z\vee \ce)\we (w\vee \ce), z\vee \ce) \in \theta.
\]
In a similar way it can be proved that $(({\sim} z\vee \ce)\we
({\sim} w\vee \ce), \sim w\vee \ce) \in \theta$.

\begin{rem}
Let $T\in \cKl$ and $\theta \in \Con(T)$. Since the class of
centered Kleene algebras is a variety, then $T/\theta \in \cKl$. In
particular, the lattice order $\leq$
of $T/\theta$ is given by $x/\theta \leq y/\theta$ if and only if
$x/\theta = (x\we y)/\theta$. In this framework
 the relation $\ll$ given in Definition \ref{wbo}
 coincides with the relation $\leq$, i.e.,
\[
x/\theta \leq y/\theta\;\textrm{if and only if}\; x/\theta \ll
y/\theta.
\]
To prove it  note first that from the distributivity of the
underlying lattice of $T$  it follows that $x/\theta \leq
y/\theta$ if and only if $(x \vee \ce, (x\we y)\vee \ce) \in
\theta$ and $(x\we \ce, (x\we y)\we \ce) \in \theta$. Besides, we
have that $(x\we y) \vee \ce = (x\vee \ce)\we (y\vee \ce)$. Since
$\theta$ preserves the involution, then $(x\we \ce, (x\we y) \we
\ce) \in \theta$ if and only if $({\sim} x \vee \ce, {\sim} x \vee
{\sim} y \vee \ce) \in \theta$. Therefore
\begin{equation} \label{oi1}
x/\theta \leq y/\theta\;\textrm{if and only if}\;((x\vee \ce)\we
(y\vee \ce), x\vee \ce) \in \theta\;\textrm{and}\;({\sim} x \vee
\ce, {\sim} x \vee {\sim} y \vee \ce) \in \theta.
\end{equation}
We also have
\begin{equation} \label{oi2}
({\sim} x \vee \ce, {\sim} x \vee {\sim} y \vee \ce) \in
\theta\;\textrm{if and only if}\;(({\sim} x \vee \ce) \we ({\sim}
y \vee \ce), {\sim y} \vee \ce)\in \theta.
\end{equation}
In order to prove (\ref{oi2}), suppose that $\;({\sim} x \vee \ce,
{\sim} x \vee {\sim} y \vee \ce) \in \theta$. Since $({\sim} y
\vee \ce,{\sim} y \vee \ce)\in \theta$, then taking $\we$ we obtain
that $(({\sim} x \vee \ce) \we ({\sim} y \vee \ce), {\sim y} \vee
\ce)\in \theta$. Conversely, assume that $(({\sim} x \vee \ce) \we
({\sim} y \vee \ce), {\sim y} \vee \ce)\in \theta$. Since $({\sim}
x \vee \ce, {\sim} x \vee \ce) \in \theta$, then taking $\vee$ we
obtain that $({\sim} x \vee \ce, {\sim} x \vee {\sim} y \vee \ce)
\in \theta$, so $({\sim} x \vee {\sim} y \vee \ce, {\sim} y \vee
\ce) \in \theta$. Then we have proved (\ref{oi2}). Therefore, it
follows from (\ref{oi1}) and (\ref{oi2}) that $x/\theta \leq
y/\theta$ if and only if $x/\theta \ll y/\theta$.
\end{rem}

\begin{lem} \label{rc1}
Let $T \in \KhIS$ and $\theta$ a well-behaved congruence of $T$.
Then $(T,\ll)$ is a poset.
\end{lem}

\begin{proof}
Let $\theta$ be a well-behaved congruence of $T$. The reflexivity
of $\theta$ implies the reflexivity of $\ll$. In order to prove
that $\ll$ is antisymmetric, let $x, y\in T$ be such that $x/\theta
\ll y/ \theta$ and $y/\theta \ll x/\theta$, which means that
\[
((x\vee \ce) \we (y\we \ce), x\vee \ce) \in \theta,
\]
\[
(({\sim} y \vee \ce)\we ({\sim} x \vee \ce), {\sim} y \vee \ce)
\in \theta,
\]
\[
((y\vee \ce) \we (x\vee \ce), y\vee \ce) \in \theta,
\]
\[
(({\sim} x \vee \ce)\we ({\sim} y \vee \ce), {\sim} x \vee \ce)
\in \theta.
\]
Since $(x\vee \ce, (x\vee \ce)\we (y\vee \ce)) \in \theta$ and
$((x\vee \ce)\we (y\vee \ce), y\vee \ce) \in \theta$, then $(x\vee
\ce, y\vee \ce) \in \theta$. Analogously we have that $({\sim} x
\vee \ce, {\sim} y \vee \ce) \in \theta$. Hence, it follows from
$\Ttw$ that $(x,y)\in \theta$, i.e., $x/\theta = y/\theta$. We
conclude that  $\ll$ is antisymmetric. Finally we will prove that
$\ll$ is transitive. Let $x$, $y$ and $z$ be elements of $T$ such
that $x/\theta \ll y/\theta$ and $y/\theta \ll z/\theta$. In
particular,
\begin{equation} \label{EQ1}
((x\vee \ce) \we (y\vee \ce), x\vee \ce)\in \theta,
\end{equation}
\begin{equation} \label{EQ2}
((y\vee \ce)\we (z\vee \ce), y\vee \ce) \in \theta.
\end{equation}
It follows from (\ref{EQ1}) and $\Tth$ that
\begin{equation} \label{EQ3}
((x\vee \ce) \we (y\vee \ce)\we (z\vee \ce), (x\vee \ce)\we (z\vee
\ce)) \in \theta,
\end{equation}
and it follows from (\ref{EQ2}) and $\Tth$ that
\begin{equation} \label{EQ4}
((x\vee \ce)\we (y\vee \ce)\we (z\vee \ce), (x\vee \ce)\we (y\vee
\ce)) \in \theta.
\end{equation}
Hence, by (\ref{EQ3}) and (\ref{EQ4}) we obtain that $((x\vee
\ce)\we (y\vee \ce), (x\vee \ce)\we (z\vee \ce)) \in \theta$.
Thus, taking into account (\ref{EQ1}) we have  $((x\vee
\ce)\we (z\vee \ce), x\vee \ce) \in \theta$. Similarly we can show
that $(({\sim} z \vee \ce)\we ({\sim} x \vee \ce), {\sim} z \vee
\ce) \in \theta$. Thus, $x/\theta \ll z/\theta$. Hence, $\ll$ is
transitive.
\end{proof}

\begin{lem} \label{rc2}
Let $T\in \KhIS$ and $x,y \in T$. If $x\leq y$, then $x/\theta \ll
y/\theta$.
\end{lem}

\begin{proof}
Let $x\leq y$. Then ${\sim} y \leq {\sim} x$. Hence, we have
$x\vee \ce \leq y \vee \ce$ and ${\sim} y \vee \ce \leq {\sim} x
\vee \ce$, i.e., $(x\vee \ce)\we (y\vee \ce) = x\vee \ce$ and
$({\sim} y \vee \ce) \we ({\sim} x \vee \ce) = {\sim} y \vee \ce$.
Since $\theta$ is a reflexive relation, then $((x\vee \ce) \we
(y\we \ce), x\vee \ce) \in \theta$ and $(({\sim} y \vee \ce)\we
({\sim} x \vee \ce), {\sim} y \vee \ce) \in \theta$, i.e.,
$x/\theta \ll y/\theta$.
\end{proof}

For $T\in \KhIS$ and $\theta$ a well-behaved congruence of $T$, we
have in particular that $\theta$ is a congruence of
$(T,{\sim},\ra)$. Let us  use also the symbols ${\sim}$ and $\ra$ to refer to the
respective induced operations on $T/\theta$.

\begin{prop}
Let $(T,\leq,{\sim},\ra,\ce,0,1)\in \KhIS$ and $\theta$ a
well-behaved congruence of $T$. Then $(T/\theta,\ll,{\sim},\ra,
\ce/\theta, 0/\theta, 1/\theta) \in \KhIS$.
\end{prop}

\begin{proof}
\textsc{Step 1}. $(T/\theta,\ll,{\sim},\ce/\theta) \in \KPo$.

It follows from Lemma \ref{rc1} that $(T/\theta,\ll)$ is a poset.
It is immediate that ${\sim}$ is an involution in $(T/\theta,\ll)$
which is order reversing and that ${\sim}\ce/\theta = \ce/\theta$.
Let $x\in T$. In what follows we will prove that the supremum of
$x/\theta$ and $\ce/\theta$ with respect to the order $\ll$ exists
in $T/\theta$, and we will denote it by $x/\theta \vee
\ce/\theta$. Moreover, we will prove that $x/\theta \vee
\ce/\theta = (x\vee \ce)/\theta$. First note that if $y\in
x/\theta$, then it follows from $\Ttw$ that $(x\vee \ce,y\vee \ce)
\in \theta$, i.e., that $(x\vee \ce)/\theta = (y\vee \ce)/\theta$.
Now we will show that $x/\theta \vee \ce/\theta$ exists. Since
$x\leq x\vee \ce$ and $\ce\leq x \vee \ce$, it follows from Lemma
\ref{rc2} that $x/\theta \ll (x\vee \ce)/\theta$ and $\ce/\theta
\ll (x\vee \ce)/\theta$. Let $z\in T$ be such that $x/\theta \ll
z/\theta$ and $\ce/\theta\ll z/\theta$. Then $((x\vee
\ce)\we(z\vee \ce), x\vee \ce)\in \theta$, $(({\sim}z \vee
\ce)\we({\sim}x \vee \ce), {\sim}z \vee \ce) \in \theta$ and
$(c,{\sim} z \vee \ce)\in \theta$. We need to prove that $(x\vee
\ce)/\theta \ll z/\theta$. By the previous assertions we have in
particular that
\begin{equation} \label{sup1}
(((x\vee \ce)\vee \ce)\we(z\vee \ce), (x\vee \ce) \vee \ce)\in
\theta.
\end{equation}
On the other hand,
\[
({\sim}z \vee \ce)\we ({\sim}(x\vee \ce)\vee \ce) = \ce.
\]
But $(c,{\sim}z \vee \ce) \in \theta$, so
\begin{equation} \label{sup2}
(({\sim}z \vee \ce)\we ({\sim}(x\vee \ce)\vee \ce), {\sim}z \vee
\ce) \in \theta.
\end{equation}
Hence, it follows from (\ref{sup1}) and (\ref{sup2}) that $(x\vee
\ce)/\theta \ll z/\theta$. Thus, $x/\theta \vee \ce/\theta$ exists
and $x/\theta \vee \ce/\theta = (x\vee \ce)/\theta$.

In what follows we will prove that for every $x,y \in T$,
\[
(x/\theta\vee \ce/\theta)\we ({\sim}x/\theta \vee \ce/\theta) =
\ce/\theta,
\]
or, equivalently, that
\begin{equation} \label{center}
(x\vee \ce)/\theta \we ({\sim} x \vee \ce)/\theta = \ce/\theta,
\end{equation}
where we also use $\we$ for the infimum with respect to $\ll$.
In order to prove (\ref{center}), note that it follows from Lemma
\ref{rc2} that $c/\theta \ll (x\vee \ce)/\theta$ and $\ce/\theta
\ll ({\sim} x \vee \ce)/\theta$. Let $z\in T$ such that $z/\theta
\ll (x\vee \ce)/\theta$ and $z/\theta \ll ({\sim}x \vee
\ce)/\theta$. In particular,
\begin{equation} \label{sup3}
((z\vee \ce) \we (x\vee \ce), z \vee \ce) \in \theta,
\end{equation}
\begin{equation} \label{sup4}
((z\vee \ce) \we ({\sim} x \vee \ce), z\vee \ce) \in \theta.
\end{equation}
It follows from $\Tth$, (\ref{sup3}) and (\ref{sup4}) that
\begin{equation}
((z\vee \ce) \we (x\vee \ce) \we ({\sim} x \vee \ce) , z \vee \ce) \in \theta.
\end{equation}
Since $(x\vee \ce) \we ({\sim} x \vee \ce) = \ce$, then $(\ce,
z\vee \ce) \in \theta$, i.e., $z/\theta \ll \ce/\theta$.
Therefore, $(x/\theta\vee \ce/\theta)\we ({\sim}x/\theta \vee
\ce/\theta) = \ce/\theta$.

For $x,y\in T$ assume that $(x\vee \ce)/\theta \ll (y\vee \ce)/\theta$
and $(x\we \ce)/ \theta \ll (y\we \ce)/\theta$. It is immediate that
$x/\theta \ll y/\theta$. Then we conclude that
$(T/\theta,\ll,{\sim},\ce/\theta) \in \KPo$.

\textsc{Step 2}. $(T/\theta,\ll,{\sim},\ce/\theta,0/\theta,
1/\theta) \in \KMS$.

Since for every $x\in T$ we have  $0\leq x \leq 1$,  it
follows from Lemma \ref{rc2} that $0/\theta \ll x/\theta \ll
1/\theta$, i.e., $0/\theta$ is the first element of
$(T/\theta,\ll)$ and $1/\theta$ is the last element of
$(T/\theta,\ll)$.

Let $x$ and $y$ be elements of $T$. Recall that it follows from
$\KMth$ that $(x\vee \ce) \we y$ exists. We will prove that $(x
\vee \ce)/\theta \we y/\theta$ exists and is $((x\vee \ce)\we
y)/\theta$. In order to do it,  we will prove first that if $(x,z)
\in \theta$ and $(y,w) \in \theta$, then $((x\vee \ce)\we
y)/\theta = ((z\vee \ce)\we w)/\theta$.

Let $(x,z) \in \theta$ and $(y,w) \in \theta$. It follows
from $\Ttw$ that $(x\vee \ce,z\vee \ce) \in \theta$ and $(y\vee
\ce,w\vee \ce) \in \theta$. By $\Tth$ we have that
\begin{equation} \label{EQ5}
((x\vee \ce) \we (y\vee \ce),(z\vee \ce) \we (w\vee \ce))\in
\theta.
\end{equation}
Taking into account $\KMfo$ we also have
\begin{equation} \label{EQ6}
((x\vee \ce)\we y) \vee \ce = (x\vee \ce)\we (y\vee \ce),
\end{equation}
\begin{equation} \label{EQ7}
((z\vee \ce)\we w) \vee \ce = (z\vee \ce)\we (w\vee \ce).
\end{equation}
Hence, it follows from (\ref{EQ5}), (\ref{EQ6}) and (\ref{EQ7})
that
\begin{equation} \label{EQ7-1}
(((x\vee \ce)\we y) \vee \ce, ((z\vee \ce)\we w) \vee \ce)\in
\theta.
\end{equation}
In a similar way, taking into account that $({\sim} x\vee
\ce,{\sim} z\vee \ce) \in \theta$ and $({\sim} y\vee \ce,{\sim}
w\vee \ce) \in \theta$ we have
\begin{equation} \label{EQ7-2}
((({\sim} x\vee \ce)\we {\sim} y) \vee \ce, (({\sim} z\vee \ce)\we
{\sim} w) \vee \ce)\in \theta.
\end{equation}
Then it follows from (\ref{EQ7-1}), (\ref{EQ7-2}) and $\Ttw$ that
\[
((x\vee \ce)\we y, (z\vee \ce)\we w)\in \theta.
\]

Now we will prove that $(x/\theta \vee \ce/\theta) \we y/\theta$
exists and is $((x\vee \ce)\we y)/\theta$. This is equivalent to
prove that $(x \vee \ce)/\theta \we y/\theta$ exists and is
$((x\vee \ce)\we y)/\theta$. Since $(x\vee \ce)\we y \leq x\vee
\ce$ and $(x\vee \ce)\we y \leq y$, then it follows from Lemma
\ref{rc2} that $((x\vee \ce)\we y)/\theta \ll (x\vee \ce)/\theta$
and $((x\vee \ce)\we y)/\theta \ll y/\theta$. Let $z\in T$ be such
that $z/\theta \ll (x\vee \ce)/\theta$ and $z/\theta \ll
y/\theta$. In particular,
\begin{equation} \label{EQ8}
((z\vee \ce)\we (x\vee \ce), z\vee \ce)\in \theta,
\end{equation}
\begin{equation} \label{EQ9}
((z\vee \ce)\we (y\vee \ce), z\vee \ce)\in \theta,
\end{equation}
\begin{equation} \label{EQ10}
(({\sim} y\vee \ce)\we ({\sim} z\vee \ce), {\sim} y\vee \ce)\in
\theta.
\end{equation}
We need to prove that $z/\theta \ll ((x\vee \ce)\we y)/\theta$,
which means that
\begin{equation} \label{EQ11}
((z\vee \ce)\we (((x\vee \ce)\we y) \vee \ce), z\vee \ce) \in
\theta
\end{equation}
and
\begin{equation} \label{EQ12}
({\sim}((x\vee \ce)\we y) \vee \ce) \we ({\sim z} \vee \ce),
{\sim}((x\vee \ce)\we y) \vee \ce) \in \theta.
\end{equation}
It follows from $\KMfo$ that
\begin{equation} \label{EQ12-1}
(z\vee \ce)\we (((x\vee \ce)\we y) \vee \ce) = (z\vee \ce)\we
(y\vee \ce)\we (x\vee \ce),
\end{equation}
and it follows from (\ref{EQ8}) and $\Tth$ that
\begin{equation} \label{EQ12-2}
((z\vee \ce)\we (y\vee \ce)\we (x\vee \ce), (z\vee \ce)\we (y\vee
\ce))\in \theta.
\end{equation}
Thus, by (\ref{EQ9}), (\ref{EQ12-1}), and (\ref{EQ12-2}) we obtain
(\ref{EQ11}). It is immediate that the condition (\ref{EQ12}) is
equal to the condition (\ref{EQ10}) because
\[
{\sim}((x\vee \ce)\we y) \vee \ce) \we ({\sim z} \vee \ce) =
({\sim} y\vee \ce)\we ({\sim} z\vee \ce),
\]
\[
{\sim}((x\vee \ce)\we y) \vee \ce = {\sim} y\vee \ce.
\]
Then $(T/\theta,\ll,{\sim},\ce/\theta,0/\theta, 1/\theta)$
satisfies $\KMth$. The condition $\KMfo$ follows from the previous
steps and from the same condition on $T$. In consequence, we
obtain that $(T/\theta,\ll,{\sim},\ce/\theta,0/\theta,
1/\theta)\in \KMS$.

\textsc{Step 3}. $(T/\theta,\ll,{\sim},\ra,\ce/\theta,0/\theta,
1/\theta) \in \KhIS$. The other conditions to be an object of
$\KhIS$ follow from the previous steps, the fact that $T\in \KhIS$
and Lemma \ref{rc2}.
\end{proof}

In what follows we will study the lattice of well-behaved
congruences of any object of $\KhIS$. We start with some
preliminary definitions. Let $T\in \KhIS$. Recall that it follows
from previous results of this paper that $\C(T) \in \hIS$. Note
that $T$ does not necessarily satisfy the condition ($\CK$). We
write $\Conwb(T)$ to refer both to the set and to the lattice of
well-behaved congruences of $T$. For $\theta \in \Conwb(T)$ we
define the binary relation $\Gamma(\theta)$ on $\C(T)$ as the
restriction of $\theta$ to $\C(T) \times \C(T)$. For $\tau \in
\Con(\C(T))$ we define the relation $\Sigma{(\tau)} \subseteq T
\times T$ in the following way:
\[
\text{$(x,y) \in \Sigma(\tau)$ if and only if $(x\vee \ce,y\vee
\ce) \in \tau$ and $({\sim} x\vee \ce,{\sim} y\vee \ce) \in
\tau$.}
\]
We prove that $\Sigma(\tau)$ is a well behaved congruence of
$T$.

\begin{lem}\label{lem:Sigma-tau}
Let $T\in \KhIS$ and $\tau \in \Con(\C(T))$. Then $\Sigma(\tau)
\in \Conwb(T)$.
\end{lem}

\begin{proof}
Let $\tau \in \Con(\C(T))$. Straightforward computations show
that $\Sigma(\tau)$ satisfies $\Ttw$. In order to show that
$\Sigma(\tau)$ satisfies $\Tth$, let $x$, $y$, $z$ and $w$ in
$\C(T)$ be such that $(x,y) \in \Sigma(\tau)$ and $(z,w) \in
\Sigma(\tau)$, which means that $(x,y) \in \tau$ and $(z,w) \in
\tau$. Then $(x\we z,y\we w)\in \tau$, because $\tau \in
\Con(\C(T))$. But $(x\we z) \vee \ce = x\we z$ and $(y\we w) \vee
\ce = y\we w$. Thus,
\[
((x\we y) \vee \ce, (z\we w)\vee \ce) \in \tau.
\]
On the other hand, since ${\sim} (x\we z) \vee \ce = \ce$ and
${\sim} (y\we w) \vee \ce = \ce$, then
\[
({\sim} (x\we z) \vee \ce, {\sim} (y\we w) \vee \ce) \in \tau.
\]
Hence, $(x\we z, z\we w) \in \Sigma(\tau)$, so the condition
$\Tth$ holds. Now we show the condition $\Ton$. It is immediate
that $\Sigma(\tau)$ is congruence with respect to ${\sim}$.

In order to prove that $\Sigma(\tau)$ is congruence with
respect to $\ra$, let $(x,y) \in \Sigma(\tau)$ and $(z,w) \in
\Sigma(\tau)$, so
\begin{equation} \label{eq1}
(x\vee \ce, y\vee \ce) \in \tau,
\end{equation}
\begin{equation} \label{eq2}
(z\vee \ce, w\vee \ce) \in \tau,
\end{equation}
\begin{equation} \label{eq3}
({\sim} x\vee \ce, {\sim} y\vee \ce) \in \tau,
\end{equation}
\begin{equation} \label{eq4}
({\sim} z\vee \ce, {\sim} w \vee \ce) \in \tau.
\end{equation}
Then taking $\we$ in (\ref{eq1}) and (\ref{eq4}) we have
\begin{equation} \label{eq5}
((x\vee \ce) \we ({\sim} z \vee \ce), (y\vee \ce)\we ({\sim} w
\vee \ce)) \in \tau.
\end{equation}
But it follows from $\Kfo$ that ${\sim}(x\ra z) \vee \ce = (x\vee
\ce) \we ({\sim} z \vee \ce)$ and ${\sim}(y \ra w) \vee \ce =
(y\vee \ce)\we ({\sim} w \vee \ce)$. So by (\ref{eq5}) we obtain
that
\begin{equation}\label{eq6}
({\sim}(x\ra z) \vee \ce, {\sim}(y\ra w) \vee \ce) \in \tau.
\end{equation}
On the other hand, taking $\ra$ between (\ref{eq1}) and
(\ref{eq2}) we have that
\begin{equation}\label{eq7}
((x\vee \ce) \ra (z\vee \ce), (y\vee \ce) \ra (w\vee \ce))\in
\tau,
\end{equation}
and taking $\ra$ between (\ref{eq4}) and (\ref{eq3}) we obtain
\begin{equation}\label{eq8}
(({\sim} z\vee \ce) \ra ({\sim} x\vee \ce), ({\sim} w\vee \ce) \ra
({\sim} y\vee \ce))\in \tau.
\end{equation}
Define now the following elements:
\[
t := ((x\vee \ce)\ra (z\vee \ce))\we (({\sim} z\vee \ce) \ra
({\sim} x \vee \ce)),
\]
\[
u := ((y\vee \ce)\ra (w\vee \ce))\we (({\sim} w\vee \ce)\ra ({\sim}
y \vee \ce)).
\]
Taking $\we$ in (\ref{eq7}) and (\ref{eq8}) we obtain that
\begin{equation}\label{eq9}
(t,u)\in \tau.
\end{equation}
Besides, it follows from $\Kfi$ that
\begin{equation} \label{eq10}
(x\ra z) \vee \ce = ((x\vee \ce) \ra (z\vee \ce)) \we (({\sim} z
\vee \ce) \ra ({\sim} x \vee \ce)),
\end{equation}
\begin{equation}\label{eq11}
(y\ra w) \vee z = ((y\vee \ce)\ra (w\vee \ce)) \we (({\sim} w \vee
\ce)\ra ({\sim} y \vee \ce)).
\end{equation}
Taking into account (\ref{eq9}), (\ref{eq10}), and (\ref{eq11}) we
have
\begin{equation} \label{eq12}
((x\ra z)\vee \ce, (y\ra w)\vee \ce) \in \tau.
\end{equation}
Thus, by (\ref{eq6}) and (\ref{eq12}) the condition $(x\ra z, y\ra
w) \in \Sigma(\tau)$ is satisfied. This implies that $\Sigma(\tau)
\in \Conwb(T)$.
\end{proof}

\begin{prop} \label{lemc1}
Let $T\in \KhIS$. There exists an isomorphism between $\Conwb(T)$
and $\Con(\C(T))$, which is established via the assignments
$\theta \mapsto \Gamma(\theta)$ and $\tau \mapsto \Sigma(\tau)$.
\end{prop}

\begin{proof}
Let $\theta\in \Conwb(T)$. It follows from $\Ton$ and $\Tth$ that
$\Gamma(\theta) \in \Con (\C(T))$. Suppose now that $\theta\in
\Conwb(T)$, $\sigma \in \Conwb(T)$ and $\Gamma(\theta) =
\Gamma(\sigma)$. Let $(x,y) \in \theta$. Then by $\Ttw$ we have
$(x\vee \ce, y\vee \ce) \in \theta$ and $({\sim} x \vee \ce,
{\sim} y \vee \ce) \in \theta$, so $(x\vee \ce, y\vee \ce) \in
\Gamma(\theta)$ and $({\sim}x \vee \ce, {\sim}y \vee \ce) \in
\Gamma(\theta)$. Since $\Gamma(\theta) = \Gamma(\sigma)$, $(x\vee
\ce, y\vee \ce) \in \sigma$ and $({\sim}x \vee \ce, {\sim} y \vee
\ce) \in \sigma$. Hence, it follows from $\Ttw$ again that $(x,y)
\in \sigma$. Thus, $\theta \subseteq \sigma$. For the same reason
we have the other inclusion, so $\theta = \sigma$.

Lemma~\ref{lem:Sigma-tau} shows that if $\tau \in \Con(\C(T))$,
then $\Sigma(\tau)\in \Conwb(T)$. Besides it is immediate that
$\Gamma(\Sigma(\tau)) = \tau$. We also have that for $\theta\in
\Conwb(T)$ and $\sigma \in \Conwb(T)$, $\theta \subseteq \sigma$
if and only if $\Sigma(\theta) \subseteq \Sigma (\sigma)$.
Therefore, we obtain an isomorphism between $\Conwb(T)$ and
$\Con(\C(T))$.
\end{proof}

Let $T\in \KhBDL$. If $\theta \in \Con(T)$ and $\tau \in \C(T)$,
we define $\Gamma(\theta)$ and $\Sigma(\tau)$ as for the case of
$\KhIS$. If $\theta \in \Con(T)$, then $\theta$ satisfies $\Ton$,
$\Ttw$, and $\Tth$. Let $\tau \in \C(T)$. The distributivity of the
underlying lattice of $T$ proves that $\Sigma(\tau)$ preserves
$\we$ and $\vee$. Then  from the proof of Proposition
\ref{lemc1} the next result follows.

\begin{prop} \label{lemc1b}
Let $T\in \KhBDL$. There exists an isomorphism between $\Con(T)$
and $\Con(\C(T))$, which is established via the assignments
$\theta \mapsto \Gamma(\theta)$ and $\tau \mapsto \Sigma(\tau)$.
\end{prop}

Let $H \in \hIS$ or $H \in \hBDL$. Let $\theta \in \Con(H)$ and
$\tau \in \Con(\C(\K(H)))$. Since the map $\alpha:H\ra \C(K(H))$
given by $\alpha(a) = (a,0)$ is an isomorphism, we have that the
binary relation $\alpha(\theta) = \{(\alpha(a),\alpha(b)): (a,b)
\in \theta\}$ in $\C(\K(H))$
is a congruence of $\C(\K(H))$. Moreover, the relation
$\alpha^{-1}(\tau)$ in $H$ given by $(a,b) \in \alpha^{-1}(\tau)$
if and only if $((a,0), (b,0)) \in \tau$ is a congruence of $H$.
Then the following result follows from propositions \ref{lemc1}
and \ref{lemc1b}.

\begin{cor}
\begin{enumerate} [\normalfont (a)]
\item[]
\item Let $H\in \hIS$. There exists an isomorphism between
$\Con(H)$ and $\Conwb(\K(H))$, which is established via the
assignments $\theta \mapsto \Sigma(\alpha(\theta))$ and $\tau
\mapsto \alpha^{-1}(\Gamma(\tau))$.
\item Let $H\in \hBDL$. There
exists an isomorphism between $\Con(H)$ and $\Con(\K(H))$, which
is established via the assignments $\theta \mapsto
\Sigma(\alpha(\theta))$ and $\tau \mapsto
\alpha^{-1}(\Gamma(\tau))$.
\end{enumerate}
\end{cor}

\begin{rem}
Let $H\in \hIS$, $\theta \in \Con(H)$ and $\tau \in
\Conwb(\C(\K(H)))$. Then
\[
((a,b),(d,e)) \in \Sigma(\alpha(\theta))\; \textrm{if and only if}
\; (a,d) \in \theta\; \textrm{and}\; (b,e) \in \theta,
\]
\[
(a,b) \in \alpha^{-1}(\Gamma(\tau))\; \textrm{if and only if}\;
((a,0),(b,0)) \in \tau.
\]
Similarly for $H\in \hBDL$.
\end{rem}

\subsection{Well-behaved congruences and congruences: the relation
with some family of filters and some applications}

We start by recalling some facts about congruences in $\hIS$ and
congruences in $\hBDL$ \cite{SM2}. Let $H\in \hIS$ or $H\in
\hBDL$. As usual, we say that $F$ is a \emph{filter} if it is a
nonempty subset of $H$ which satisfies the following conditions:
\begin{enumerate}
 \item If $a\in F$ and $b\in F$ then $a\we b \in
F$. \item If $a\in F$ and $a\leq b$ then $b\in F$.
\end{enumerate}
We also consider the binary relation associated with $F \subseteq H$
\[
\Theta(F) = \{(a,b) \in H\times H: a\we f = b \we f\;\text{for
some}\; f\in F\}.
\]
Note that if $H$ is an upper bounded semilattice and $F$ is a
filter, then $\Theta(F)$ is a congruence.
Let $H\in \hIS$ or $H\in \hBDL$. For $a,b,f\in H$ we
define the following element of $H$:
$$t(a,b,f): = (a \ra b) \rla ((a\we f) \ra (b\we f)),$$ where $a\rla
b:= (a\ra b)\we (b\ra a)$.

The next definition was introduced in \cite{SM2}.

\begin{defn}
Let $H\in \hIS$ or $H\in \hBDL$, and let $F$ be a filter of $H$.
We say that $F$ is a \emph{congruent filter} if $t(a,b,f)\in F$
whenever $a$, $b \in H$ and $f\in F$.
\end{defn}

Note that the set of all congruent filters of $H\in \hIS$ or of $H\in
\hBDL$ is closed under arbitrary intersections and therefore for
every $X \subseteq H$ the congruent filter generated by $X$
exists.

\begin{rem}
Let $F$ be a congruent filter of a hemi-implicative semilattice
(lattice). We will see that $(a,b) \in \Theta(F)$ if and only if
$a\rla b \in F$. In order to show it, suppose that $a\rla b \in
F$. Since $a\we (a\rla b) = b \we(b\rla a)$, then $(a,b) \in
\Theta(F)$. Conversely, assume that $(a,b) \in \Theta(F)$, i.e.,
$a\we f = b\we f$ for some $f\in F$. Since $t(a,b,f) \in F$ and
$t(a,b,f) = (a\ra b) \rla 1$, then $1\ra (a\ra b) \in F$ because
$(a\ra b) \rla 1 \leq 1\ra (a\ra b)$. Since $1\ra (a\ra b) \leq
a\ra b$, then $a\ra b\in F$. In a similar way we can show that
$b\ra a\in F$. Hence, $a\rla b\in F$. Thus,
\[
\Theta(F) = \{(a,b)\in H\times H: a\rla b \in F\}.
\]
\end{rem}

The following result was proved in \cite{SM2}.

\begin{thm} \label{teofc}
Let $H\in \hIS$ or $H\in \hBDL$. There exists an isomorphism
between $\Con(H)$ and the lattice of congruent filters of $H$,
which is established via the assignments $\theta \mapsto 1/\theta$
and $F\mapsto \Theta(F)$.
\end{thm}

Taking into account Theorem \ref{teofc}, it is possible to show
that Proposition \ref{lemc1b} can be seen as a corollary of
Proposition \ref{lemc1}. In order to show this assertion, let $T_1
= (T,\we,\vee,\ra,{\sim},\ce,0,1) \in \hBDL$. Then we write $T_2 =
(T,\leq,{\sim},\ra,\ce,0,1)$ for the corresponding object of
$\KhIS$. Since the set of congruent filters of $\C(T_1)$ is equal
to the set of congruent filters of $\C(T_2)$, then it follows from
Theorem \ref{teofc} that $\Con(\C(T_1)) = \Con(\C(T_2))$. In what
follows we will see that $\Con(T_1) = \Conwb(T_2)$. It is
immediate that $\Con(T_1) \subseteq \Conwb(T_2)$. Conversely, let
$\theta \in \Conwb(T_2)$. We will prove that $\theta$ preserves
$\we$ and $\vee$. Let $(x,y)\in \theta$ and $(z,w) \in \theta$.
Then it follows from $\Ttw$ that $(x\vee \ce, z\vee \ce) \in
\theta$ and $(y \vee \ce,w\vee \ce) \in \theta$. Then by $\Tth$ we
have that $((x\vee \ce)\we (z \vee \ce), (y\vee \ce)\we (w\vee
\ce))\in \theta$. But by the distributivity of the underlying
lattice of $T_1$ we deduce that $(x\we z) \vee \ce = (x\vee
\ce)\we (z\vee \ce)$ and $(y\we w) \vee \ce = (y\vee \ce)\we
(w\vee \ce)$. Thus,
\begin{equation} \label{cwbc0}
((x\we z)\vee \ce, (y\we w)\vee \ce \in \theta.
\end{equation}
Besides, since $(x,y)\in \theta$ and $(z,w)\in \theta$, then it
follows from the condition $\Ttw$ that $({\sim} x \vee \ce,{\sim}
y \vee \ce) \in \theta$ and $({\sim} y \vee \ce,{\sim} w \vee \ce)
\in \theta$. Equivalently, we have that
\begin{equation} \label{cwbc1}
({\sim} x \vee \ce,{\sim} y \vee \ce) \in \Gamma(\theta),
\end{equation}
\begin{equation} \label{cwbc2}
({\sim} y \vee \ce,{\sim} w \vee \ce) \in \Gamma(\theta).
\end{equation}
Since $\Gamma(\theta) \in \Con(\C(T_2))$ and $\Con(\C(T_1)) =
\Con(\C(T_2))$, then taking $\vee$ in (\ref{cwbc1}) and
(\ref{cwbc2}) we obtain $({\sim} x \vee {\sim} y \vee \ce, {\sim}
z \vee {\sim} w \vee \ce) \in \theta$, i.e.,
\begin{equation} \label{cwbc4}
({\sim}(x\we z)\vee \ce, {\sim}(z\we w) \vee \ce)\in \theta.
\end{equation}
Then it follows from $\Ttw$, (\ref{cwbc0}) and (\ref{cwbc4}) that
$(x\we z, y\we w)\in \theta$. The same argument combined with
$\Ton$ proves that $({\sim} x \we {\sim} z, {\sim} y \we {\sim} w)
\in \theta$, so $(x\vee z,y\vee w)\in \theta$. Hence,  $\theta$
preserves $\we$ and $\vee$, which implies that $\theta \in
\Con(T_1)$. Then $\Con(T_1) = \Conwb(T_2)$. Therefore, since
$\Con(T_1) = \Conwb(T_2)$ and $\Con(\C(T_1)) = \Con(\C(T_2))$, we
deduce that Proposition \ref{lemc1b} can be seen as a corollary of
Proposition \ref{lemc1}.

\begin{cor} \label{propc2}
Let $T\in \KhIS$. There exists an isomorphism between $\Conwb(T)$
and the lattice of congruent filters of $\C(T)$, which is
established via the assignments $\theta \mapsto 1/\Gamma(\theta)$
and $F\mapsto \Sigma(\Theta(F))$.
\end{cor}

\begin{proof}
It follows from Proposition \ref{lemc1} and Theorem \ref{teofc}.
\end{proof}

Similarly, the following result follows from Proposition
\ref{lemc1b} and Theorem \ref{teofc}.

\begin{cor} \label{propc2b}
Let $T\in \KhBDL$. There exists an isomorphism between $\Con(T)$
and the lattice of congruent filters of $\C(T)$, which is
established via the assignments $\theta \mapsto 1/\Gamma(\theta)$
and $F\mapsto \Sigma(\Theta(F))$.
\end{cor}

For implicative semilattices Corollary \ref{propc2} can be
simplified, and for semi-Heyting algebras Corollary \ref{propc2b}
also can be simplified. More precisely: if $H \in \IS$ or $H\in
\SH$ then the congruent filters of $H$ are all the filters of $H$
\cite{SM2}.

Let $H\in \Hil$ and $F\subseteq H$. Recall that $F$ is said to be
a \textit{deductive system} \cite{D} if the following conditions
are satisfied: a) $1\in F$, b) if $a\in F$ and $a\ra b\in F$ then
$b\in F$. Also recall that a deductive system $F$ is said to be
\emph{absorbent} \cite{Fig} if $a \ra (a\we b)\in F$ whenever
$a\in F$. It follows from Theorem \ref{teofc} and \cite[Lemma
3.3]{Fig} that the congruent filters of $H$ are the absorbent
deductive systems of $H$.

\begin{defn}
Let $A$ be an algebra and $a_1,b_1,\dots,a_n,b_n$ elements of $A$.
We write $\theta_{A}((a_1,b_1),\dots,(a_n,b_n))$ for the
congruence generated by $(a_1,b_1),\dots,(a_n,b_n)$. If $T\in
\KhIS$ we also write $\theta_{T}((a_1,b_1),\dots,(a_n,b_n))$ for
the well-behaved congruence generated by
$(a_1,b_1),\dots,(a_n,b_n)$.
\end{defn}

Let $H\in \hIS$ or $H\in \hBDL$, and let $a \in H$. We refer by
$F^{c}(a)$ to the congruent filter generated by $\{a\}$. In
\cite{SM2} the following assertions were proved:
\begin{enumerate}
\item  if $H\in \hIS$ or $H\in \hBDL$, then $(d,e) \in
\theta_{H}(a,b)$ if and only if $d\rla e \in F^{c}(a\rla b)$;
\item  if $H\in \IS$ or $H\in \SH$ then $(d,e) \in
\theta_{H}(a,b)$ if and only if $a\rla b \leq d\rla e$.
\end{enumerate}

 The following remark will be used
later.

\begin{rem} \label{clase1}
Let $H\in \hIS$ or $H\in \hBDL$. Let $\tau \in \Con(H)$. Then
$(a,b) \in \tau$ if and only if $a\rla b \in 1/\tau$. Moreover,
$(a,b)$, $(d,e) \in \tau$ if and only if $(a\rla b)\we (d\rla e)
\in 1/\tau$.
\end{rem}

In what follows we describe some aspects of the principal
well-behaved congruences of the objects of $\KhIS$ and some
aspects of the principal congruences of the  algebras in $\KhBDL$.
Let $T\in \KhIS$ or $T\in \KhBDL$. For $x$ and $y$ elements of $T$
we also write $x\rla y$ for the element $(x\ra y)\we (y\ra x)$.

\begin{lem}\label{lpc1}
Let $T\in \KhIS$ or $T\in \KhBDL$. Let $x$, $y$, $z$, $w \in T$.
Then
\begin{enumerate}[\normalfont (a)]
\item $(z,w) \in \theta_{T}(x,y)$ if and only if $(z\vee \ce,w\vee
\ce) \in \theta_{\C(T)}((x\vee \ce, y\vee \ce),(\sim x\vee \ce,
\sim y\vee \ce))$ and $(\sim z\vee \ce,\sim w\vee \ce) \in
\theta_{\C(T)}((x\vee \ce, y\vee \ce),(\sim x\vee \ce, \sim y\vee
\ce))$. \item  If $x$, $y$, $z$ and $w$ are in $\C(T)$, then
\[
1/\theta_{\C(T)}((x,y),(z,w)) = F^{c}((x\rla y) \we (z\rla w)).
\]
\end{enumerate}
\end{lem}

\begin{proof}
We  consider $T\in \KhIS$ (the proof for $T\in \KhBDL$ is
analogous).

First we  prove a). Let $(z,w) \in \theta_{T}(x,y)$. Then
$(z,w) \in \theta$ for every $\theta \in \Conwb(T)$ such that
$(x,y) \in \theta$. Now we  see that
\[
(z\vee \ce,w\vee \ce) \in \theta_{\C(T)}((x\vee \ce, y\vee
\ce),(\sim x\vee \ce, \sim y\vee \ce)),
\]
\[
(\sim z\vee \ce,\sim w\vee \ce) \in \theta_{\C(T)}((x\vee \ce,
y\vee \ce),(\sim x\vee \ce, \sim y\vee \ce)).
\]
Let $\tau \in \Con(\C(T))$ such that $(x\vee \ce, y\vee \ce) \in
\tau$ and $(\sim x \vee \ce, \sim y \vee \ce) \in \tau$. It
follows from Proposition \ref{lemc1} that $\Sigma(\tau) \in
\Conwb(T)$. We also have that $(x,y) \in \Sigma(\tau)$. Then by
hypothesis we obtain that $(z,w) \in \Sigma(\tau)$. Hence, $(z\vee
\ce, w\vee \ce) \in \tau$ and $(\sim z \vee \ce, \sim w \vee \ce)
\in \tau$, which was our aim.

Conversely, assume that $(z\vee \ce,w\vee \ce) \in
\theta_{\C(T)}((x\vee \ce, y\vee \ce),(\sim x\vee \ce, \sim y\vee
\ce))$ and $(\sim z\vee \ce,\sim w\vee \ce) \in
\theta_{\C(T)}((x\vee \ce, y\vee \ce),(\sim x\vee \ce, \sim y\vee
\ce))$. Let $\theta \in \Conwb(T)$ be such that $(x,y) \in \theta$.
It follows from Proposition \ref{lemc1} that $\Gamma(\theta) \in
\Con(\C(T))$. Moreover, $(x\vee \ce,y\vee \ce) \in \Gamma(\theta)$
and $(\sim x\vee \ce,\sim y\vee \ce) \in \Gamma(\theta)$. Thus by
hypothesis we have that $(z\vee \ce, w\vee \ce) \in
\Gamma(\theta)$ and $(\sim z\vee \ce, \sim w\vee \ce) \in
\Gamma(\theta)$, i.e., $(z\vee \ce, w\vee \ce) \in \theta$ and
$(\sim z\vee \ce, \sim w\vee \ce) \in \theta$. Then it follows
from $\Ttw$ that $(z,w) \in \theta$. Thus, $(z,w) \in
\theta_{T}(x,y)$.

Finally, we  prove b). Let $H\in \hIS$. We write $\tau$ for an
arbitrary well-behaved congruence of $H$. Then
\[
\theta_{H}((x,y),(z,w)) = \bigcap\{\tau \in \Conwb(H): (x,y),(z,w) \in \tau\}.
\]
Hence,
\[
1/\theta_{H}((x,y),(z,w)) = \bigcap\{1/\tau: \tau \in \Conwb(H) \text{ and } (x,y), (z,w) \in \tau\}.
\]
Then it follows from Remark \ref{clase1} that
\[
1/\theta_{H}((x,y),(z,w)) = \bigcap\{ 1/\tau: \tau \in \Conwb(H) \text{ and }  (x\rla y)\we (z\rla w) \in 1/\tau\}.
\]
Thus, by Theorem \ref{teofc} we have that
\[
1/\theta_{H}((x,y),(z,w)) = F^{c}((x\rla y)\we (z\rla w)).
\]
In particular, the last assertion holds for $H = \C(T)$.
\end{proof}

\begin{rem}
The proof of item (b) of Lemma \ref{lpc1} can be done in a
different way. Let $\theta$ be a congruence of an algebra $H\in
\hIS$, and let $a,b\in H$. Straightforward computations show that
$F^{c}(a) \vee F^{c}(b) = F^{c}(a \we b)$, where $\vee$ is the
supremum in the lattice of congruent filters of $H$. On the other
hand, it follows from general results from universal algebra that
$\theta_{H}((x,y),(z,w)) = \theta_{H}(x,y) \vee \theta_H(z,w)$,
where $\vee$ is the supremum in the lattice of congruences of $H$
\cite{Sanka}. In \cite{SM2} it was proved that $1/\theta_{H}(x,y)
= F^{c}(x\rla y)$. Then
\[
\begin{array}
[c]{lllll}
1/\theta_{H}((x,y),(z,w)) & = &  1/\theta_{H}(x,y) \vee 1/\theta_{H}(z,w) &  & \\
& = & F^{c}(x\rla y) \vee F^{c}(z\rla w)&  & \\
& = & F^{c}((x\rla y) \we (z\rla w)).& &
\end{array}
\]
\end{rem}

Let $T\in \KhIS$ or $T\in \KhBDL$. For every $x$, $y\in T$ we
define the following binary term:
\[
q(x,y) = ((x\vee \ce) \rla (y\vee \ce))\we ((\sim x\vee \ce) \rla
(\sim y\vee \ce)).
\]

In the proof of the following corollary we will use Remark
\ref{clase1} and Lemma \ref{lpc1}.

\begin{cor}
Let $T \in \KhIS$ or $T\in \KhBDL$. Let $x$, $y$, $z$, $w \in T$.
\begin{enumerate}[\normalfont (a)]
\item $(z,w) \in \theta_{T}(x,y)$ if and only if $q(z,w) \in
F^{c}(q(x,y))$. \item If $T\in \KIS$ or $T \in \KSH$ then $(z,w)
\in \theta_{T}(x,y)$ if and only if $q(x,y) \leq q(z,w)$.
\end{enumerate}
\end{cor}

\begin{proof}
The condition $(z,w) \in \theta_{T}(x,y)$ is equivalent to the
conditions
\[
(z\vee \ce,w\vee \ce) \in \theta_{\C(T)}((x\vee \ce, y\vee
\ce),(\sim x\vee \ce, \sim y\vee \ce)),
\]
\[
(\sim z\vee \ce,\sim w\vee \ce) \in \theta_{\C(T)}((x\vee \ce,
y\vee \ce),(\sim x\vee \ce, \sim y\vee \ce)),
\]
which are equivalent to
\[
(z\vee \ce) \rla (w\vee \ce) \in 1/\theta_{\C(T)}((x\vee \ce,
y\vee \ce),(\sim x\vee \ce, \sim y\vee \ce)),
\]
\[
({\sim} z\vee \ce) \rla ({\sim} w\vee \ce) \in
1/\theta_{\C(T)}((x\vee \ce, y\vee \ce),(\sim x\vee \ce, \sim
y\vee \ce)),
\]
which happens if and only if
\[
q(z,w) \in 1/\theta_{\C(T)}((x\vee \ce, y\vee \ce),(\sim x\vee
\ce, \sim y\vee \ce)).
\]
But this last fact is equivalent to say that $q(z,w) \in
F^{c}(q(x,y))$.

If $T \in \KIS$ or $T\in \KSH$,  then $F^{c}(q(x,y))$ is equal to
the filter generated by $\{q(x,y)\}$, so $(z,w) \in
\theta_{T}(x,y)$ if and only if $q(x,y) \leq q(z,w)$.
\end{proof}

\vspace{3mm}

\subsection*{Acknowledgments}

This project has received funding from the European Union’s
Horizon 2020 research and innovation programme under the Marie
Sklodowska-Curie grant agreement No 689176.

The first author was also partially supported by the research
grant 2014 SGR 788 from the government of Catalonia and by the
research projects MTM2011-25747 and MTM2016-74892-P from the
government of Spain, which include \textsc{feder} funds from the
European Union. The second author was also supported by CONICET
Project PIP 112-201501-00412, and he thanks Marta Sagastume for
several conversations concerning the matter of this paper.


---------------------------------------------------------------------------------------
\\
Ramon Jansana,\\
Departament de Filosofia,\\
Universitat de  Barcelona.\\
Montalegre, 6,\\
08001, Barcelona,\\
España.\\
jansana@ub.edu

-----------------------------------------------------------------------------------------
\\
Hern\'an Javier San Mart\'in,\\
Departamento de Matem\'atica, \\
Facultad de Ciencias Exactas (UNLP), \\
and CONICET.\\
Casilla de correos 172,\\
La Plata (1900),\\
Argentina.\\
hsanmartin@mate.unlp.edu.ar

\end{document}